\documentclass[11pt]{article}
\usepackage[margin=1in]{geometry}
\usepackage{amsmath}
\usepackage{amsfonts}
\usepackage{color}
\usepackage{latexsym}
\usepackage{tablefootnote}
\usepackage{epsfig}
\usepackage{wrapfig}
\usepackage{multirow}
\usepackage{mathtools}
\mathtoolsset{showonlyrefs=true}

\usepackage{tikz}
\usetikzlibrary{matrix,arrows.meta}
\usepackage{subcaption}
\usetikzlibrary{calc}
\usepackage{float}
\usepackage{array}
\usepackage{hyperref}
\usepackage{mathtools}
\usepackage{amssymb,amsthm}
\usepackage{algorithm,algorithmic}
\theoremstyle{definition}

\newcommand*{\QEDA}{\hfill\hbox{\vrule width1.0ex height1.0ex}}

\newcommand{\Argmax}{\operatorname{Argmax}}
\usepackage{makecell,booktabs}
\newcommand{\tO}{\tilde{\mathcal{O}}}
\usepackage[version=4]{mhchem}
\usepackage[strict]{changepage}

\usepackage{amsmath}
\usepackage{amsfonts}
\usepackage{latexsym}
\usepackage{amssymb}

\newtheorem{thm}{Theorem}[section]
\newtheorem{theorem}[thm]{Theorem}

\newtheorem{lemma}[thm]{Lemma}

\newtheorem{proposition}[thm]{Proposition}

\newtheorem{definition}[thm]{Definition}

\newtheorem{remark}[thm]{Remark}

\newcommand{\beq}{\begin{equation}}
\newcommand{\eeq}{\end{equation}}
\newcommand{\beqa}{\begin{eqnarray}}
\newcommand{\eeqa}{\end{eqnarray}}
\newcommand{\beqas}{\begin{eqnarray*}}
\newcommand{\eeqas}{\end{eqnarray*}}
\newcommand{\bi}{\begin{itemize}}
\newcommand{\ei}{\end{itemize}}

\newcommand{\R}{\mathbb{R}}

\newcommand{\lam}{{\lambda}}

\newcommand{\inner}[2]{\langle #1,#2\rangle}

\newcommand{\argmin}{\mathrm{argmin}\,}
\newcommand{\argmax}{\mathrm{argmax}\,}

\newcommand{\dom}{\mathrm{dom}\,}

\newcommand{\Argmin}{\mathrm{Argmin}\,}

\newcommand{\tx}{\tilde x}

\newcommand{\D}{\operatorname{D}}
\begin{document}
	\title{Accuracy Certificates for Convex Optimization at Accelerated Rates via Primal-Dual Averaging}
	\date{April 20, 2026}
    \ifdefined\anon
    \author{Anonymous Authors}
    \else
	\author{Matthew X. Burns\thanks{
        Department of Electrical and Computer Engineering, University of Rochester, Rochester, NY 14627 (email: {\tt mburns13@ur.rochester.edu}).}
  \and
		Jiaming Liang \thanks{Goergen Institute for Data Science and Artificial Intelligence and Department of Computer Science, University of Rochester, Rochester, NY 14620 (email: {\tt jiaming.liang@rochester.edu}). This work was supported by AFOSR grant FA9550-25-1-0182.
		}
		 }
         
    \fi
	\maketitle
    \begin{abstract}
        Many works in convex optimization provide rates for achieving a small primal gap. However, this quantity is typically unavailable in practice. In this work, we show that solving a regularized surrogate with algorithms based on simple primal-dual averaging provides non-asymptotic convergence guarantees for a \textit{computable} optimality certificate. We first analyze primal and dual methods based on one average, namely modified dual averaging and generalized conditional gradient, and establish $\tO(\varepsilon^{-1})$ certificate complexities. Motivated by asymmetries in the one-average case, we analyze a self-dual, two-average method that preserves symmetry while losing certificate guarantees. To recover certificate convergence, we propose a three-average method that achieves an accelerated $\tO(\varepsilon^{-1/2})$ certificate complexity. Furthermore, we prove primal-dual algorithm correspondences for the one, two, and three-average cases. In particular, the primal three-average accelerated method mirrors the well-known gradient extrapolation method in the dual. By interpreting our results through the lens of zero-sum matrix games and Fisher markets, we further connect primal-dual averaging methods to game theory and market dynamics.\\
        

        \ifdefined\anon
        \else

        \fi
	\end{abstract}

    \section{Introduction}
    Continuous convex optimization has become ubiquitous in large-scale computing, with applications in machine learning~\cite{Lan2020,shalev2012online,sra2011optimization,wright2022optimization}, compressive sensing~\cite{candes2006robust,candes2008exact}, statistical inference~\cite{juditsky2020statistical}, and imaging science~\cite{beckFastIterativeShrinkageThresholding2009,chambolle2011first} (to name a few). In several settings, we require solutions to satisfy a defined optimality bound. Since we cannot run methods indefinitely, it is crucial to have some \textit{certificate} of optimality. In this work, we show that simple methods based on regularization and primal-dual averaging can provide provably convergent, computable optimality certificates. 
    Our focus is the convex smooth composite optimization (CSCO) problem
    \begin{equation}\label{def:probIntro}
        \phi_*=\min_{x\in\R^n}\{\phi(x):=f(x)+h(x)\},
    \end{equation}
    where $f$ is a closed proper convex function that is $L$-smooth with respect to the primal norm $\|\cdot\|$, and $h$ is merely closed proper convex with bounded domain. For fixed $\varepsilon>0$, we say that a point $x\in\R^n$ is an $\varepsilon$-solution if it achieves an $\varepsilon$-small primal gap, i.e., $\phi(x) - \phi_* \le \varepsilon$.

    Instead of solving~\eqref{def:probIntro} directly, our strategy is to solve the regularized problem
    \begin{equation}\label{def:probIntro_reg}
        \phi_*^\alpha=\min_{x\in\R^n}\{\phi^\alpha(x):=f(x)+h(x)+\alpha w(x)\},
    \end{equation}
    where $w:\R^n\to[0,+\infty]$ is non-negative\footnote{This assumption can be made without loss of generality by translation, since strong-convexity implies that $w$ is bounded from below on $\dom h$.}, closed, and $1$-strongly convex with respect to $\|\cdot\|$ on $\dom h$ and satisfies the following: (1) the generalized linear minimization oracle (GLMO) 
        \begin{equation}
            \min_{x\in\R^n}\left\{\inner{v}{x}+h(x)+\alpha w(x)\right\},
        \end{equation}
        is efficiently computable for all $\alpha>0$ and $v\in\R^n$, and (2)
        $w$ is bounded on $\dom h$ with $M:=\max_{x\in\dom h}w(x)<\infty$.

            
    By choosing $\alpha = \mathcal{O}(\varepsilon/M)$ and solving \eqref{def:probIntro_reg} to $\mathcal{O}(\varepsilon)$ accuracy, we can obtain an $\mathcal{O}(\varepsilon)$ solution to~\eqref{def:probIntro} (see Lemma~\ref{lem:reg_connection}). Adding $\mathcal{O}(\varepsilon)$ regularization can have additional computational and theoretical benefits, such as parallel computation in discrete optimal transport~\cite{cuturiSinkhornDistancesLightspeed2013}, nonergodic convergence rates in bilinear saddle-point problems~\cite{cen2021fast}, and more stable variable selection in sparse regression~\cite{zou2005regularization}. 
    
    Recent work~\cite{gutmanPerturbedFenchelDuality2023} demonstrates that we can obtain stronger primal-dual convergence results by considering the Fenchel-Rockafellar dual to~\eqref{def:probIntro_reg},
    
\begin{equation}\label{eq:fenchel_rockafellar_dual}
        \phi^\alpha_*=\max_{z\in\R^n}\{-(h^{\alpha})^*(-z)-f^*(z)\}=-\min_{z\in\R^n}\{\psi^\alpha(z):=(h^{\alpha})^*(-z)+f^*(z)\}=-\psi^\alpha_*,
    \end{equation}
     where we define the aggregate function $h^\alpha(\cdot)=h(\cdot)+\alpha w(\cdot)$ for convenience and $f^*$ is the convex conjugate of $f$. By our assumptions on $f$, $h$, and $w$, we have that $(h^\alpha)^*$ is $(1/\alpha)$-smooth and $f^*(\cdot)$ is $(1/L)$-strongly convex relative to the dual norm $\|\cdot\|_*$. 

    Numerous prior works have revealed deep connections between seemingly disparate optimization algorithms in the CSCO setting by demonstrating that they are ``dual'' to each other: one algorithm solving the primal~\eqref{def:probIntro_reg} generates the same iterate sequences as another algorithm solving the dual~\eqref{eq:fenchel_rockafellar_dual}. Duality correspondences often lead to simplified convergence proofs~\cite{lu2021generalized, tibshiraniDykstraAlgorithmADMM2017}, optimality certificate guarantees~\cite{bach2015duality}, and novel algorithm variants~\cite{tibshiraniDykstraAlgorithmADMM2017,burns2026improved,wangNoregretDynamicsFenchel2023a}. Existing primal-dual algorithm pairs include generalized conditional gradient (GCG) and mirror descent~\cite{bach2015duality}, ADMM and Dykstra's algorithm~\cite{tibshiraniDykstraAlgorithmADMM2017}, cyclic coordinate descent and ADMM~\cite{tibshiraniDykstraAlgorithmADMM2017}, a stochastic variant of GCG and randomized coordinate descent~\cite{lu2021generalized}, and a primal-dual cutting-plane method and GCG~\cite{fersztand2024proximal,liangPrimaldualProximalBundle2025a}. Beyond the setting of~\eqref{def:probIntro}, duality correspondences can provide more efficient algorithms in constrained problems~\cite{burns2026improved} and one-level convex reformulations in bilevel optimization~\cite{aboussoror2011fenchel}.

    \begin{table}[]
        \centering
        \begin{tabular}{lccccc}
            \toprule
            \textbf{Averages} & \textbf{Algorithm}&  \textbf{Ref} & \textbf{Perspective} & \textbf{Complexity} & \textbf{Cert. Complexity}
            \\
            \midrule
            \multirow{ 2}{*}{1} & MDA & \cite{nesterov2009primal,liangPrimaldualProximalBundle2025a} & Primal & $\tO(\varepsilon^{-1})$ & $\tO(\varepsilon^{-1})$\\
            & GCG &  \cite{brediesGeneralizedConditionalGradient2009} & Dual & $\tO(\varepsilon^{-1})$ & $\tO(\varepsilon^{-1})$\\
            \midrule
            \multirow{ 1}{*}{2} & AggGCG & \cite{zhao2023generalized} & Primal/Dual & $\tO(\varepsilon^{-1})$ & -\\
            \midrule
            \multirow{ 2}{*}{3} & TAA & \textbf{This Work} & Primal & $\tO(\varepsilon^{-1/2})$ & $\tO(\varepsilon^{-1/2})$\\
            & GEM & \cite{lan2018random} & Dual & $\tO(\varepsilon^{-1/2})$ & $\tO(\varepsilon^{-1/2})$\\
            \bottomrule
        \end{tabular}
        \caption{High-level summary of primal-dual algorithm pairs considered in this work. We view a ``Primal'' algorithm as solving problem~\eqref{def:probIntro_reg} and a ``Dual'' algorithm as solving problem~\eqref{eq:fenchel_rockafellar_dual}. ``Complexity'' refers to the complexity of computing an $\varepsilon$-solution to~\eqref{def:probIntro}, while ``Cert. Complexity'' refers to the complexity of computing a verifiable primal-dual certificate of $\varepsilon$-optimality (See Definition~\ref{def:pd_cert}). Acronyms are as follows: ``MDA'' is ``Modified Dual Averaging'', ``GCG'' is ``Generalized Conditional Gradient'', ``AggGCG'' is ``Aggregated GCG'', ``TAA'' is ``Three-Average Acceleration'', and ``GEM'' is the ``Gradient Extrapolation Method''.}
        \label{tab:intro}
    \end{table}
    \noindent\textbf{Our contributions} In this work, we characterize primal-dual pairs in three algorithmic classes targeting the regularized problem~\eqref{def:probIntro_reg}, methods with one, two, and three averages, along with their complexity and optimality certificate guarantees. The algorithms considered are summarized in Table~\ref{tab:intro}. For one average, we show that a modified dual averaging (MDA) method in the primal is equivalent to GCG in the dual. Furthermore, both algorithms provide computable primal-dual $\varepsilon$-optimality certificates with $\tO(\varepsilon^{-1})$ complexity. For two averages, we show that a previously proposed method \cite{zhao2023generalized}, which we term aggregated GCG (AggGCG), is self-dual. While the algorithm possesses appealing symmetry, it does not admit the same straightforward primal-dual certificate convergence as the one-average case. Finally, we propose a three-average accelerated (TAA) method. TAA recovers the computable primal-dual certificate of the one-average case with an accelerated $\tO(\varepsilon^{-1/2})$ complexity. Furthermore, an interesting finding is that, through duality, TAA is equivalent to a variant of the well-known gradient extrapolation method (GEM) \cite{lan2018random}. As an additional benefit, we show that the GEM also admits a computable optimality certificate with $\tO(\varepsilon^{-1/2})$ complexity, a novel result in the CSCO setting. We provide additional insight into our results by game-theoretic interpretations for each pair of primal-dual correspondences and Fisher market interpretations for each of the three averaging methods in the primal space.

    \noindent\textbf{Outline} Section~\ref{sec:prelims} provides the key definitions and running examples. Section~\ref{sec:one_avg} considers the simple case of one-average methods, providing primal-dual equivalence results and $\tO(\varepsilon^{-1})$ certificate complexities for MDA and GCG. Moving to two averages, Section~\ref{sec:two_avg} formalizes the self-duality of AggGCG along with a $\tO(\varepsilon^{-1})$ primal gap complexity bound for~\eqref{def:probIntro}. Motivated by AggGCG's lack of computable certificate guarantees, Section~\ref{sec:three_avg} proposes the TAA method. Along with its certificate complexity, we further state its formal equivalence to the GEM. Conclusions and directions for future work are discussed in Section~\ref{sec:conclusion}. Technical proofs and self-contained analyses for each method are deferred to the appendices.

    \section{Preliminaries}\label{sec:prelims}
    Let $\R^n$ be the $n$-dimensional Euclidean space equipped with the standard inner product $\inner{\cdot}{\cdot}$ and norm $\|\cdot\|$. The $n$-dimensional unit simplex is given by $\Delta^n$.  For a closed convex function $f$, we denote the \textit{subdifferential} of $f$ at $x$ by $\partial f(x)$, and the \textit{linearization} of $f$ at $x_0\in\dom f$ as $
        \ell_f(\cdot;x_0):=f(x_0)+\inner{f'(x_0)}{\cdot-x_0},$ where $f'(x_0)\in\partial f(x_0)$. If $f$ is continuously differentiable, then $\partial f(x)=\{\nabla f(x)\}$.
    For a convex function $f$, we define the \textit{Bregman divergence} of $f$ as $\D_f(x\|y):=f(x)-\ell_f(x;y)$. We define the \textit{convex conjugate} of a closed and proper function $f$ as
    $f^*(y)=\max_{x\in\R^n}\{\inner{y}{x}-f(x)\}$.
    The conjugate $f^*$ is convex, and for a closed proper convex $f$ we have the identity $f=(f^*)^*$ (see~\cite[Theorems 4.3 and 4.8]{Beck2017}). We define the \textit{domain} of $f$ as $\dom f=\{x\in\R^n:f(x)<\infty\}$.
For $\mu\geq 0$, we say that $f$ is $\mu$\textit{-strongly convex} on $Q\subseteq\dom f$ with respect to the norm $\|\cdot\|$ if 
    $\D_f(x\|y) \geq \mu\|x-y\|^2/2$ for all $x,y\in Q$. Additionally, for $L>0$ we say that $f$ is $L$-\textit{smooth} with respect to the norm $\|\cdot\|$ on $Q$ if $\D_f(x\|y) \leq L\|x-y\|^2/2$ for all $x,y\in Q$.
    

    With basic definitions established, we motivate the substitution of~\eqref{def:probIntro_reg} to solve~\eqref{def:probIntro} by connecting a primal-dual solution of the regularized problem~\eqref{def:probIntro_reg} to a primal solution of the original objective~\eqref{def:probIntro}. See Appendix~\ref{appdx:technical} for the proof.
    \begin{lemma}\label{lem:reg_connection}
        Assume that $w:\R^n\to\R_+$ is a non-negative function and $M:=\max_{x\in\dom h}w(x)<\infty$ is bounded. Given $\varepsilon>0$ and choosing $\alpha \leq \varepsilon/(2M)$, if the primal-dual pair $(x,s)\in\dom h\times\R^n$ satisfies $\phi^\alpha(x)+\psi^\alpha(s)\leq\varepsilon/2$, then we have $\phi(x)-\phi_*\leq \varepsilon$.
    \end{lemma}

    Now we define a function that will be critical in our definition of an optimality certificate.
    \begin{definition}[Aggregated Cutting Plane (ACP) Model]\label{def:acp_model}
        Given an initial point $y_0$, a set of points $\{x_i\}_{i=0}^{k-1}\in\dom \phi$ and a set of convex combination parameters $\zeta\in[0,1]^k$, we define the associated ACP model as
        \begin{equation}\label{def:ACP_model}
            \Gamma_0(x)=h^\alpha(x)+\ell_f(x;y_0),\quad \Gamma_{j+1}(x)=(1-\zeta_j)\Gamma_j(x)+\zeta_j(h^\alpha(x)+\ell_f(x;x_{j})),
        \end{equation}
        where $0\leq j\leq k-1$. We say that $\Gamma_k$ is the ACP model induced by $(y_0,\{x_i\}_{i=0}^{k-1},\zeta)$. If $k=0$, then we simply say that $\Gamma:=\Gamma_0$ is the single-cutting plane (SCP) model induced by $y_0$.
    \end{definition}
    The ACP model has been used in prior work on proximal bundle methods for convex nonsmooth optimization~\cite{liang2024unified,liang2024single}, however it has several appealing properties for broader convex optimization, as we summarize in the following lemma (whose proof is deferred to Appendix~\ref{appdx:technical}).
    \begin{lemma}\label{lem:model_props}
        Let $\Gamma_k(\cdot)$ be the ACP model for~\eqref{def:probIntro_reg} induced by $(y_0,\{x_i\}_{i=0}^{k-1},\zeta)$ for $y_0\in\dom h$, $\{x_i\}_{i=0}^{k-1}\subseteq \dom h$ and $\zeta\in[0,1]^k$. Then, the following hold:
        \begin{itemize}
            \item[{\rm a)}] for all $x\in\dom h$, $\Gamma_k(x)\leq \phi^\alpha(x)$;
            \item[{\rm b)}] $\Gamma_k$ is $\alpha$-strongly convex.
        \end{itemize}
        Furthermore, define $\{s_k\}_{k\ge 0}$ as 
            \begin{equation}\label{def:sk1}
                s_{0}=\nabla f(y_0),\quad s_{j+1}=(1-\zeta_j)s_j+\zeta_j \nabla f(x_{j})
            \end{equation}
            for $0\leq j \leq k-1$.
        Then, the following bound holds for all $u\in\dom h$
        \begin{itemize}
            \item[{\rm c)}] $\phi^\alpha(u) + \psi^\alpha(s_k)\leq \phi^\alpha(u)-\min_{x\in\R^n}\Gamma_k(x)$.
        \end{itemize}
    \end{lemma}

    Motivated by Lemma~\ref{lem:model_props}(c), we define a computable primal-dual optimality certificate. We define the certificate from the primal perspective, however Lemma~\ref{lem:model_props_dual}(c) shows that we can equivalently define the certificate from the dual perspective.
    \begin{definition}[Primal-Dual Certificate]\label{def:pd_cert}
       Let $u\in \dom h$ and suppose $\alpha\leq\varepsilon/(2M)$. We say $(u,\Gamma_k)$ is an $\varepsilon$-certificate for~\eqref{def:probIntro} if
        \begin{equation}\label{ineq:pd_cert_def}
            \phi^\alpha(u)-\min_{x\in\R^n}\Gamma_k(x)\leq \frac{\varepsilon}2,
        \end{equation}
    where $\Gamma_k$ is an ACP model induced by $(y_0,\{x_i\}_{i=0}^{k-1}, \zeta)$. 
    \end{definition}

    By Lemmas~\ref{lem:reg_connection} and~\ref{lem:model_props}(c), an $\varepsilon$-certificate $(u,\Gamma_k)$ implies that $u$ is an $\varepsilon$-solution to~\eqref{def:probIntro}. Optimality certificates based on the minimum of aggregated linearizations have been used repeatedly in convex optimization~\cite{nemirovskiAccuracyCertificatesComputational2010a,rodomanovSubgradientEllipsoidMethod2023,grimmerPrimaldualTheorySubgradient2025}, known as ``accuracy certificates." Our definition of a certificate aligns with these works, unifying recent literature on proximal bundle methods~\cite{liang2024unified,liang2024single,liangPrimaldualProximalBundle2025a} and classical accuracy certificates.

    Throughout the work, we provide additional motivation and intuition for the primal-dual correspondences by adopting a game-theoretic lens. Specifically, we consider an entropically smoothed, zero-sum matrix game
\begin{equation}\label{def:bilinear_game}
    \min_{x\in\R^n}\max_{z\in\R^n}\left\{\inner{x}{Az} + \delta_{\Delta^n}(x)-\delta_{\Delta^n}(z)+\alpha H(x)-L^{-1} H(z)\right\},
\end{equation}
where $H(x)=\inner{x}{\log x}$ is the negative entropy function and we assume that the payoff matrix $A\in\R^{n\times n}$ has unit matrix norm for convenience. Computing the primal $\phi^\alpha$ and dual $d^\alpha$ functions\footnote{See~\cite[Section 4.4.10 and Theorem 4.14]{Beck2017}} gives
\begin{align}
   \phi^\alpha(x)&= \overbrace{L^{-1}\log\left(\sum_{i=1}^n\exp[LA_{:,i}^\top x]\right)}^{f(A^\top x)} 
     + \overbrace{\delta_{\Delta^n}(x)}^{h(x)}+\alpha\overbrace{H(x)}^{w(x)} \label{def:game_primal},\\
    d^\alpha(z)&= -\alpha\log\left(\sum_{j=1}^n\exp[-\alpha^{-1}A_{j,:}^\top z]\right) -L^{-1}H(z)- \delta_{\Delta^n}(z),\label{def:game_dual}
\end{align}
where $A_{:,i}$ and $A_{j,:}$ represent the $i^{\rm th}$ column and $j^{\rm th}$ row of $A$, respectively. Defining $\psi^\alpha$ as in~\eqref{eq:fenchel_rockafellar_dual} and applying standard conjugate calculus, we have \begin{equation}\label{def:game_dual_fenchel}
d^\alpha(z)=-(h^\alpha)^*(-Az)-f^*(z)=-\psi^\alpha(z),
\end{equation} 
where $f$ and $h^\alpha$ are as in \eqref{def:game_primal}.
 
In this setting, we denote $x\in\Delta^n$ and $z\in\Delta^n$ as mixed strategies of the minimizing and maximizing players, respectively, while gradients $\nabla_x f(A^\top x)$ and $\nabla_z (h^\alpha)^*(-Az)$ are predicted opponent responses to the primal and dual players, respectively. Since the primal~\eqref{def:game_primal} and dual~\eqref{def:game_dual} functions are negative Fenchel dual to each other by~\eqref{def:game_dual_fenchel}, the prediction exactly accords with reality. For simplicity, we henceforth use the terms ``player prediction'' and ``opponent response'' interchangeably without reiterating their equivalence. Moreover, we can show that the functions~\eqref{def:game_primal} and~\eqref{def:game_dual} satisfy the assumptions of~\eqref{def:probIntro_reg} and~\eqref{eq:fenchel_rockafellar_dual}, so all convergence analysis for the general CSCO problem~\eqref{def:probIntro} can be translated to game setting.
    
    \section{One-Average Methods}\label{sec:one_avg}
    
    We begin with a simple baseline consisting of one GLMO and one averaging step. When the averaging step is performed over gradients, we obtain the DA~\cite{nesterov2009primal} variant shown in Algorithm~\ref{alg:PDCP}. 

    \vspace{-.2in}
    
    \begin{minipage}[t]{0.44\textwidth}
\begin{algorithm}[H]\caption{MDA}\label{alg:PDCP}
        \begin{algorithmic}
        \REQUIRE given $y_0\in\dom h$, $\alpha>0$, set $\eta=\alpha/(L+\alpha)$, $s_0 = \nabla f(y_0)$, and compute
        \begin{equation*}
            x_0 = \underset{x\in\R^n}\argmin\left\{\inner{s_{0}}{x} + h^\alpha(x)\right\}.
        \end{equation*}
        \vspace{-2em}
            \FOR{$k\geq 0$}
                \STATE Compute
                \begin{equation*}
                x_{k+1}=\underset{x\in\R^n}\argmin\left\{\inner{s_{k+1}}{x} + h^\alpha(x)\right\},
                \end{equation*}
                where $
                    s_{k+1}=(1-\eta)s_k+\eta\nabla f(x_k)$.
            \ENDFOR
        \end{algorithmic}
    \end{algorithm}
\end{minipage}
\hfill
\begin{minipage}[t]{0.56\textwidth}
    \begin{algorithm}[H]\caption{GCG}\label{alg:GCG}
        \begin{algorithmic}
            \REQUIRE $z_0\in\dom f^*$, set $\eta=\alpha/(L+\alpha)$.
            \FOR{$k\geq 0$}
                \STATE  Compute
                \vspace{-1em}
                \begin{align}
                   &\bar{z}_{k}=\underset{v\in\R^n}\argmin\left\{\inner{-\nabla (h^\alpha)^*(-z_k)}{v}+f^*(v)\right\},\label{def:barz_gcg}\\
                    &z_{k+1}=\eta\bar z_{k} + (1-\eta)z_{k}. \label{def:z+}
                \end{align}
            \ENDFOR
        \end{algorithmic}
    \end{algorithm}
\end{minipage}

    \vspace{.2in}
    
   Adopting the perspective of the dual problem~\eqref{eq:fenchel_rockafellar_dual}, we can instead employ GCG as shown in Algorithm~\ref{alg:GCG}. Here, the ``primal'' variables $z_{k+1}$ are averaged while a single ``gradient'' is used to compute $\bar z_k$. Comparing the one-average methods, we obtain the following equivalence result. A similar correspondence was obtained in~\cite{liangPrimaldualProximalBundle2025a} in the case $w(\cdot)=\|\cdot-x_0\|^2/2$ for some $x_0$ in the context of proximal bundle methods. The proof is deferred to Appendix~\ref{appdx:one_avg}.
    \begin{proposition}\label{prop:1avg_duality}
        Fix $y_0\in\dom h$ and set $z_0=\nabla f(y_0)$. Then, on every iteration $k\geq0$ of Algorithm~\ref{alg:PDCP} and Algorithm~\ref{alg:GCG}, we have the following correspondence
        \begin{equation}
             \nabla f(x_k)=\bar{z}_{k},\quad x_{k}=\nabla (h^\alpha)^*(-z_k),\quad s_{k}=z_k.
        \end{equation}
    \end{proposition}
    
    In the context of the zero-sum game~\eqref{def:bilinear_game}, we can better understand the correspondence between Algorithms~\ref{alg:PDCP} and~\ref{alg:GCG} as a sequential game dynamic, summarized in the following graph
    \[
    \begin{tikzpicture}
        \node (A) [anchor=center] at ($(0,-1mm)$) {$y_0$};
        \node (B) [anchor=center] at ($(1.7,-1mm)$) {$(s_0 = z_0)$};
        \node (C) [anchor=center] at (5,0) {$(x_k\stackrel{\text{\textbf{(a)}}}=\nabla (h^\alpha)^*(-z_k))$};
        \node (D) [anchor=center] at (8.7,0) {$(\nabla f(x_k)\stackrel{\text{\textbf{(b)}}}=\bar{z}_{k})$};
        \node (E) [anchor=center] at (11.9,0) {$(s_{k+1} \stackrel{\text{\textbf{(c)}}}= z_{k+1}).$};

        \draw[arrows = {->[]},semithick] (A.east) -- ($(B.west)$);
        
        \draw[arrows = {->[]},semithick] ($(B.east)$) -- ($(C.west)+(0,-1mm)$);
        
        \draw[arrows = {->[]},semithick] ($(C.east)+(0,-1mm)$) -- ($(D.west)+(0,-1mm)$);
        \draw[semithick, ->] ($(D.east)+(0,-1mm)$) -- ($(E.west)+(0,-1mm)$);

  \draw[semithick, ->] (E.south) -- ++(0,-0.23) -| (C.south);
    \end{tikzpicture}
    \]
    In this game, the average iterates $\{z_k\}$ are mixed strategies, while $\{x_k\}$, $\{\bar z_k\}$ are responses to opponent moves.
    
    To start the game, the primal player moves first by selecting strategy $y_0$, and then predicts the dual player's response $s_0$ simultaneously with the dual player's actual response $z_0$. We recall that, by~\eqref{def:game_dual_fenchel}, each player's prediction exactly aligns with reality (i.e., the opponent's actual response) due to conjugate symmetry of the problem.

    With the game initialized, each round occurs the same way. First, in \textbf{(a)} the dual player predicts the primal response to their mixed strategy $(x_k=\nabla (h^\alpha)^*(-z_k))$. Next, in \textbf{(b)} the primal player predicts the dual response to their previous move ($\nabla f(x_k)=\bar z_k$). Finally, in \textbf{(c)} the mixed strategy $(s_{k+1}=z_{k+1})$ is updated using the last dual response and the next round begins. Each player therefore employs simple ``regularized best response + averaging'' strategies, where the regularized best response is computed via the GLMO and the primary difference is the role of the averaging step. The primal player uses averaging to smooth opponent responses, the dual player uses averaging to compute their own mixed strategy.

    With the correspondence established in Proposition~\ref{prop:1avg_duality}, we state the following certificate complexities for each 1-average method. Full proofs are deferred to Appendix~\ref{appdx:one_avg}. 
    \begin{theorem}\label{thm:da_convergence}
        For $y_0\in\dom h$, for all $k\geq 0$ define 
        $y_{k+1} = (1-\eta)y_k + \eta x_{k+1}$.
        Given $\varepsilon>0$, choosing $\alpha=\varepsilon/(2M)$, the pair $(y_k,\Gamma_k)$ is an $\varepsilon$-certificate for~\eqref{def:probIntro} 
        in $k=\tO(1+ML\varepsilon^{-1})$ iterations of Algorithm~\ref{alg:PDCP}, where $\Gamma_k$ is the ACP model for~\eqref{def:probIntro_reg} induced by $(y_0,\{x_{i}\}_{i=0}^{k-1}, \{\eta\}^k)$.
    \end{theorem}
    \noindent\textbf{Proof Sketch}. The majority of the proof of Theorem~\ref{thm:da_convergence} is to show the inequality
    \[\Gamma_{k+1}(x_{k+1})\geq (1-\eta)\Gamma_k(x_k) + \eta\phi^\alpha(x_{k+1}),\]
    which utilizes the optimality condition of $x_{k+1}$, the strong convexity of $w$, and the smoothness/convexity of $\phi^\alpha(\cdot)$. Recursively expanding and applying the definition of $y_k$ along with the convexity of $\phi^\alpha(\cdot)$ gives a linear convergence rate in the model gap with contraction factor $(1+\alpha/L)^{-1}$. Our choice of $\alpha$ and standard analysis then gives the complexity.

    \begin{theorem}\label{thm:CG-cmplx}
        Given $\varepsilon>0$, choosing $\alpha=\varepsilon/(2M)$, the pair $(z_k,\Gamma_k^*)$ is an $\varepsilon$-certificate for~\eqref{def:probIntro} 
        in $k=\tO(1+ML\varepsilon^{-1})$ iterations of Algorithm~\ref{alg:GCG}, where $\Gamma_k^*$ is the SCP model for~\eqref{eq:fenchel_rockafellar_dual} induced by $z_k$.
    \end{theorem}
\noindent\textbf{Proof Sketch}. The proof of Theorem~\ref{thm:CG-cmplx} is somewhat ``backwards'' compared to the proof of MDA, where the model gap $\phi^\alpha(y_k)-\min_{x\in\R^n}\Gamma_k(x)$ was directly bounded. First, we prove a convergence rate for the primal-dual gap $\phi^\alpha(\tilde y_k)+\psi^\alpha(z_k)$ in Proposition~\ref{prop:CG-converge} (where $\tilde{y}_k$ is defined in~\eqref{def:gcg_y}). We then use the primal-dual convergence to bound the single-cut model gap via a modification of standard results in GCG analysis (see Lemmas~\ref{lem:wolfe_gap_pd} and~\ref{lem:wolfe_gap_to_acp}).
    
    The convergence results for the one-average algorithm pair exhibit a notable asymmetry. While Theorems~\ref{thm:da_convergence} and~\ref{thm:CG-cmplx} show primal convergence in an average iterate ($y_k$ for the former, $\tilde y_k$ for the latter), this sequence never explicitly appears in either algorithm. Instead, gradient evaluations and updates occur using the ``regularized best response'' sequence $x_k=\nabla (h^\alpha)^*(-z_k)$. 

    In the next section, we consider a fully symmetric dynamic where gradients are evaluated at the true test sequence $\{y_k\}$.

    \section{A Two-Average Method}\label{sec:two_avg}
    In this section, we study a method which utilizes two averages, one in the primal and one in the dual, shown in Algorithm~\ref{alg:AggGCG_primal}, which has been previously studied in~\cite{zhao2023generalized}. Full proofs can be found in Appendix~\ref{appdx:two_avg}. Comparing to Algorithm~\ref{alg:PDCP}, there are two primary differences. First, gradients are evaluated at the smoothed sequence $\{y_k\}$ rather than the best-response sequence $\{x_k\}$. Since $\{y_k\}$ again serves as the primal test point sequence, as Theorem~\ref{thm:agg_GCG_conv} will show, Algorithm~\ref{alg:AggGCG_primal} addresses the asymmetry noted in the prior section. Second, the order of updates is ``lagged'' in comparison to Algorithm~\ref{alg:PDCP}. The dual average $s_k$, which includes $\nabla f(y_{k-1})$ instead of the available $\nabla f(y_k)$, is used to compute $x_{k+1}$. Two-average methods have also been proposed in the nonsmooth setting~\cite{nesterov2015quasi} which do not have this ``staggered'' scheme. However, the delayed updates of Algorithm~\ref{alg:AggGCG_primal} enable the use of simple smoothness inequalities in both primal and dual, as shown in Lemma~\ref{lem:gcg_per_step_duality}.

    \vspace{-.2in}
    
    \noindent
    \begin{minipage}[t]{0.5\textwidth}
        
    \begin{algorithm}[H]\caption{AggGCG (Primal)}\label{alg:AggGCG_primal}
        \begin{algorithmic}
            \REQUIRE $y_0\in \dom h$, $s_0\in\R^n$, $\eta = \alpha / (L+\alpha)$
            \FOR{$k\geq 0$}
                \STATE \textbf{1.} Compute
                \begin{align*}
                    x_{k+1}=\underset{x\in\R^n}\argmin\left\{\inner{s_{k}}{x} +  h^\alpha(x)\right\}.
                \end{align*}
                \STATE \textbf{2.} Compute
                \begin{align}
                    \label{def:y_agg_gcg}y_{k+1}&=(1-\eta)y_{k} + \eta x_{k+1},\\
                    \label{def:s_agg_gcg}s_{k+1}&=(1-\eta)s_{k} + \eta\nabla f(y_{k}).
                \end{align}
            \ENDFOR
        \end{algorithmic}
    \end{algorithm}
    \end{minipage}\hfill
    \begin{minipage}[t]{0.5\textwidth}
        
    \begin{algorithm}[H]\caption{AggGCG (Dual)}\label{alg:AggGCG_dual}
        \begin{algorithmic}
            \REQUIRE $z_0\in \dom f^*$, $v_{0}\in\R^n$, $\eta = \alpha / (L+\alpha)$
            \FOR{$k\geq 0$}
                \STATE \textbf{1.} Compute
                \begin{align}\label{def:bz_gcg_primal}
                    \bar{z}_{k+1}=\underset{z\in\R^n}\argmin\left\{-\inner{v_{k}}{z} + f^*(z)\right\}.
                \end{align}
                \STATE \textbf{2.} Compute
                \begin{align*}
                    z_{k+1}&=(1-\eta)z_{k} + \eta\bar{z}_{k+1},\\
                    v_{k+1}&=(1-\eta)v_{k} + \eta\nabla (h^\alpha)^*(-z_{k}).
                \end{align*}
            \ENDFOR
        \end{algorithmic}
    \end{algorithm}
    \end{minipage}

    \vspace{.2in}
    
    While it may not seem obvious at first glance, the update ordering makes Algorithm~\ref{alg:AggGCG_primal} \textit{self-dual}, that is, its dual algorithm utilizes exactly the same updates, as shown in Algorithm~\ref{alg:AggGCG_dual}.
    
    \begin{proposition}\label{prop:self_dual}
        Set $s_0=z_0$ and $v_0=y_0$. Then, on every iteration $k\geq 0$ of  Algorithm~\ref{alg:AggGCG_primal} and Algorithm~\ref{alg:AggGCG_dual}, we have the following correspondence
        \begin{equation}\label{corresp:two_avg}
            y_{k}=v_{k},\quad s_{k} = z_{k},\quad x_{k+1}=\nabla (h^\alpha)^*(-z_{k}),\quad \nabla f(y_{k})=\bar z_{k+1}.
        \end{equation}
    \end{proposition}
Since Algorithm~\ref{alg:AggGCG_primal} is self-dual, we only state convergence results for the primal variant for simplicity of presentation.
\begin{theorem}\label{thm:agg_GCG_conv} 
        Given $\varepsilon>0$, choosing $\alpha=\varepsilon/(2M)$, then Algorithm~\ref{alg:AggGCG_primal} computes a pair $(y_k,s_k)$ satisfying $\phi^\alpha(y_k)+\psi^\alpha(s_k)\leq \varepsilon/2$ in $k=\tO(1+ML\varepsilon^{-1})$ iterations. Consequently, the complexity to obtain an $\varepsilon$-solution to~\eqref{def:probIntro} is $\tO(1+ML\varepsilon^{-1})$.
\end{theorem}

\noindent \textbf{Proof Sketch}: Our argument is largely similar to~\cite{zhao2023generalized}. We begin by using smoothness of $f$ and $(h^\alpha)^*$ and the strong convexity of $h^\alpha$, $f^*$ to provide a series of inequalities. Exploiting the primal-dual relations between the pairs $(x_{k+1}, s_k)$ and $(y_{k},\nabla f(y_k))$ allows us to eliminate terms and, using our choice of $\alpha$, conclude with the single-step bound
\[\phi^\alpha(y_{k+1})+\psi^\alpha(s_{k+1})\leq (1-\eta)(\phi^\alpha(y_{k})+\psi^\alpha(s_{k})),\]
which provides the complexity result when recursively expanded with our choice of $\alpha$.

To make the correspondence in Proposition~\ref{prop:self_dual} concrete, we revisit the zero-sum game~\eqref{def:bilinear_game}. Unlike the one-average case, which corresponds to sequential play, the correspondence in~\eqref{corresp:two_avg} corresponds to \textit{simultaneous play}. The game begins with strategies $y_0$ and $z_0$ for the primal and dual players, respectively. Each round then proceeds in like manner, illustrated below. Updates in the same column are interpreted to occur simultaneously.
\[
\begin{tikzpicture}[baseline=(A.center)]
  \node (A) at (0,0) {$\left(z_0=s_0,\; y_0=v_0\right)$};
  \node[anchor=center] (B) at (4,0) {$\begin{pmatrix}\bar z_{k+1}=\nabla f(y_k)\\
x_{k+1}=\nabla (h^\alpha)^*(-z_k)\end{pmatrix}$};
    \node[anchor=center] (LB) at (4,0.7) {\footnotesize\textbf{(a)}};
    \node[anchor=center] (LC) at (8,0.7) {\footnotesize\textbf{(b)}};
  \node[anchor=center] (C) at (8,0) {$\begin{pmatrix}
  z_{k+1}=s_{k+1}\\
      y_{k+1}=v_{k+1}
  \end{pmatrix}$};

  \draw[arrows = {->[]},semithick] (A.east) -- (B.west);
  \draw [->, semithick] ($(B.north east) + (0,-4mm)$) to [out=10,in=170] ($(C.north west) + (0,-4mm)$);
  \draw [->, semithick] ($(C.south west) + (0,4mm)$) to [out=-170,in=-10] ($(B.south east) + (0,4mm)$);
\end{tikzpicture}
\]
 Each round begins in \textbf{(a)} with the players simultaneously predicting the opposing strategy and determining their own response based on the previous mixed strategies. The responses/predictions are then used to update the mixed strategies in \textbf{(b)}, which are then used in the next round. The simultaneous play dynamics of Algorithm~\ref{alg:AggGCG_primal} were also examined in~\cite{zhao2023generalized}, where the authors demonstrated that Algorithm~\ref{alg:AggGCG_primal} is in fact equivalent to a deterministic variant of the popular Logistic Fictitious Play~\cite{fudenberg1993learning} algorithm.

Theorem~\ref{thm:agg_GCG_conv} guarantees convergence in the primal-dual gap $\phi^\alpha(y_k)+\psi^\alpha(s_k)$, however it does not provide guarantees for a computable primal-dual certificate. While a more careful analysis may provide a suitable certificate, the difference from the one-average Algorithm~\ref{alg:PDCP} (which provides a natural certificate) is notable and admits a simple geometric explanation. In Algorithm~\ref{alg:PDCP}, evaluating the gradient at $x_k$ naturally allows us to control $\ell_f(x_{k+1};x_k)+\frac{\alpha (1-\eta)}{2\eta}\|x_k-x_{k+1}\|^2$ by smoothness and a suitable choice of $\eta$. However, for Algorithm~\ref{alg:AggGCG_primal}, there is no such trivial relationship between the linearization point $y_{k-1}$, the test point $y_{k+1}$, and the exact minimizers $x_{k+1}$ and $x_k$, as visualized in Figure~\ref{fig:AggGCG_taa_geometry}(a). We therefore cannot straightforwardly exploit the smoothness of $f$, as done in the proof of Theorem~\ref{thm:da_convergence}.

    \section{A Three-Average Method}\label{sec:three_avg}
     Motivated by the shortcomings of AggGCG, we propose TAA, and examine its dual counterpart. TAA restores convergence guarantees for a computable certificate while achieving optimal complexity like Nesterov's accelerated gradient (AG) method. Full proofs can be found in Appendix~\ref{appdx:three_avg}. 
        \begin{algorithm}[H]\caption{TAA}\label{alg:three_avg_primal}
        \begin{algorithmic}
            \REQUIRE given $y_0\in\dom h$, $L>0$, $\alpha>0$, set $s_0=\nabla f(y_0)$ and compute
            \begin{gather}
                x_0=\underset{x\in\R^n}\argmin\left\{\inner{s_0}{x}+h^\alpha(x)\right\},\quad 
                \lambda=\frac{2\alpha}{\alpha+\sqrt{\alpha^2+4L\alpha}}\label{eq:x0_and_lambda}.
            \end{gather}
            \vspace{-1.5em}
            \FOR{$k\geq 0$}
            \STATE \textbf{1.} Compute
            \vspace{-1em}
                \begin{align}
                    \tx_{k+1}&= (1-\lambda)y_k+\lambda x_k, \label{def:tx_taa} \\
                    s_{k+1}&=(1-\lambda)s_k+\lambda\nabla f(\tx_{k+1}), \label{def:s_3avg} \\
                    x_{k+1}&=\underset{x\in\R^n}\argmin\left\{\inner{s_{k+1}}{x}+h^\alpha(x)\right\}\label{def:x_3avg},\\
                    y_{k+1}&= (1-\lambda)y_k+\lambda x_{k+1}\label{def:y_3avg}.
                \end{align}
            \ENDFOR
        \end{algorithmic}
    \end{algorithm}
     Instead of linearizing at the previous test point $y_{k-1}$, TAA linearizes at a convex combination point $\tx_{k+1}$. As visualized in Figure~\ref{fig:AggGCG_taa_geometry}(b), the triangles $(y_k,y_{k+1}, \tx_{k+1})$ and $(y_k,x_{k+1}, x_{k})$ are clearly similar. Furthermore, the rescaling coefficient is dependent only on the convex combination parameter $\lambda$. As a result, we obtain a simple geometric relation when analyzing Algorithm~\ref{alg:three_avg_primal}. Note that TAA involves three averaging steps: two in the primal and one in the dual. Despite differences from existing variants of the AG method, TAA still achieves nearly optimal complexity for CSCO, as shown below.
     
 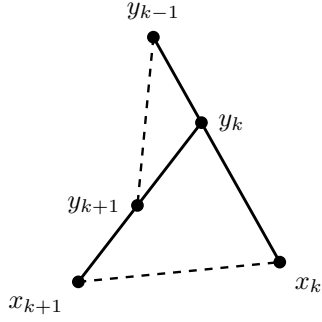
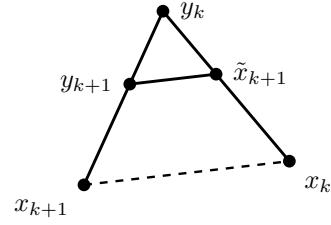
\begin{figure}[t]
\centering
\captionsetup[subfigure]{justification=centering}

\begin{subfigure}[t]{0.46\textwidth}
\centering
\begin{tikzpicture}[scale=1.05, every node/.style={font=\small}]
    \tikzset{
        pt/.style={circle, fill=black, inner sep=1.7pt},
        seg/.style={line width=1.1pt},
        aux/.style={line width=0.9pt, dashed}
    }

    \coordinate (Aykm1) at (1.1,3.2);
    \coordinate (Axk)   at (2.7,0.35);
    \coordinate (Axkp1) at (0.15,0.1);

    \coordinate (Ayk)   at ($(Aykm1)!0.38!(Axk)$);     
    \coordinate (Aykp1) at ($(Ayk)!0.52!(Axkp1)$);     

    \draw[seg] (Aykm1) -- (Axk);
    \draw[seg] (Ayk) -- (Axkp1);

    \draw[aux] (Aykm1) -- (Aykp1);
    \draw[aux] (Axkp1) -- (Axk);

    \node[pt,label=above:{$y_{k-1}$}]      at (Aykm1) {};
    \node[pt,label=right:{$y_k$}]          at (Ayk)   {};
    \node[pt,label=left:{$y_{k+1}$}]       at (Aykp1) {};
    \node[pt,label=below left:{$x_{k+1}$}] at (Axkp1) {};
    \node[pt,label=below right:{$x_k$}]    at (Axk)   {};
\end{tikzpicture}
\caption{Geometry of the AggGCG update.}
\end{subfigure}
\hfill
\begin{subfigure}[t]{0.46\textwidth}
\centering
\begin{tikzpicture}[scale=1.05, every node/.style={font=\small}]
    \tikzset{
        pt/.style={circle, fill=black, inner sep=1.7pt},
        seg/.style={line width=1.1pt},
        aux/.style={line width=0.9pt, dashed}
    }

    \coordinate (Tyk)   at (1.25,2.35);
    \coordinate (Txk)   at (2.85,0.45);
    \coordinate (Txkp1) at (0.25,0.15);
    \coordinate (Tykm1) at (0.85,3.3);

    \coordinate (Ttx)   at ($(Tyk)!0.42!(Txk)$);      
    \coordinate (Tykp1) at ($(Tyk)!0.42!(Txkp1)$);    


    \draw[seg] (Tyk) -- (Txk);
    \draw[seg] (Tyk) -- (Txkp1);
    \draw[seg] (Tykp1) -- (Ttx);
    \draw[aux] (Txkp1) -- (Txk);

    \node[pt,label=right:{$y_k$}]                 at (Tyk)   {};
    \node[pt,label=left:{$y_{k+1}$}]              at (Tykp1) {};
    \node[pt,label=right:{$\tx_{k+1}$}]     at (Ttx)   {};
    \node[pt,label=below left:{$x_{k+1}$}]        at (Txkp1) {};
    \node[pt,label=below right:{$x_k$}]           at (Txk)   {};
\end{tikzpicture}
\caption{Geometry of the TAA update.}
\end{subfigure}

\caption{Geometric illustration of the AggGCG and TAA updates. In AggGCG, there is no trivially exploitable relationship between the minimizer displacement $x_{k+1}-x_k$ and the linearization point displacement $y_{k+1}-y_{k-1}$. In TAA, the linearization point $\tx_{k+1}$ restores a simple geometric relationship by interpolating between the last test point $y_k$ and minimizer $x_k$. As a result, the magnitude of the vectors \(y_{k+1}-\tx_{k+1}\) and \(x_{k+1}-x_k\) can be related by a $\lambda$-dependent rescaling.}
\label{fig:AggGCG_taa_geometry}
\end{figure}

  \begin{theorem}\label{thm:taa_cmplx}
        
        Given $\varepsilon>0$, choosing $\alpha=\varepsilon/(2M)$, then the pair $(y_k,\Gamma_k)$ is an $\varepsilon$-certificate for~\eqref{def:probIntro} 
        in $k=\tO(1+\sqrt{ML/\varepsilon})$ iterations of Algorithm~\ref{alg:three_avg_primal}, where $\Gamma_k$ is the ACP model for~\eqref{def:probIntro_reg} induced by $(y_0,\{\tx_{i+1}\}_{i=0}^{k-1}, \{\lambda\}^k)$.
    \end{theorem}  
    \noindent\textbf{Proof Sketch:} The proof of Theorem~\ref{thm:taa_cmplx} follows a similar pattern as Theorem~\ref{thm:da_convergence}. However, instead of showing a recursive inequality in the $\{x_k\}$ sequence, we show the inequality
    \[
    \Gamma_{k+1}(x_{k+1})\geq \phi^\alpha(y_{k+1})-(1-\lambda)^{k+1}(\phi^\alpha(y_0)- \Gamma_{0}(x_{0})),
    \]
    which directly results in a convergence rate. The complexity bound follows from our choice of $\lambda$.

    Theorem~\ref{thm:taa_cmplx} demonstrates that Algorithm~\ref{alg:three_avg_primal} achieves near-optimal, accelerated complexity for solving~\eqref{def:probIntro}. However, the algorithm is substantially different from other acceleration schemes based on proximal mappings. To clarify the source of TAA's acceleration, we consider the GEM \cite{lan2018random} applied to the dual problem~\eqref{eq:fenchel_rockafellar_dual}, shown in Algorithm~\ref{alg:GEM-dual}. 
    To use the GEM, we need a suitable Bregman regularizer. Accordingly, we assume that the function $\nu^*:\R^n\to(-\infty,\infty]$ is $1$-strongly convex with respect to the dual norm $\|\cdot\|_*$ and is $L$-smooth relative to $f^*$\footnote{That is, $Lf^*(\cdot)-\nu^*(\cdot)$ is a convex function. Note that this definition does not assume that $\nu^*$ is differentiable, as in~\cite{luRelativelySmoothConvex2018}. Equivalently, we can say that $f^*$ is $L^{-1}$-strongly convex relative to $\nu^*$.} with $\dom f^*\subseteq \dom \nu^*$. Note that either $\nu^*=Lf^*$ or $\nu^*=\frac{1}{2}\|\cdot\|_*^2$ would satisfy these requirements. 
\begin{algorithm}[H]\caption{GEM}\label{alg:GEM-dual}
        \begin{algorithmic}
            \REQUIRE given $\alpha>0$ and $g_0\in\dom f^*$, set $z_0=g_0$, $v_{-1}=v_0=\nabla (h^\alpha)^*(-g_0)$, $\tau_0=\alpha/L$, $A_0=1$, and $a_{-1}=0$
            \FOR{$k\geq 0$}
            
            \STATE \textbf{1.} Compute
            \vspace{-2em}
            \begin{gather}
               \tau_{k+1}=\tau_k+\frac{\alpha a_k}{L},\quad a_k=\frac{\tau_k+\sqrt{\tau_k^2+4\tau_kA_k}}{2},\label{eq:gem_scalars}\\
             A_{k+1}=A_k+a_k,\quad    \hat v_{k} =v_k+\frac{a_{k-1}}{a_k}(v_k-v_{k-1}).\label{eq:gem_hatx}
            \end{gather}
            \vspace{-2em}
            \STATE \textbf{2.} Compute
            
            \vspace{-2em}
                \begin{align}
                    g_{k+1}&=\underset{g\in\R^n}\argmin\left\{ a_k[\inner{-\hat v_k}{g}+f^*(g)]+\frac{\tau_k}{\alpha}\D_{\nu^*}(g\|g_k)\right\}\label{def:g_gem},\\
                    z_{k+1}&=\frac{A_k}{A_{k+1}}z_{k}+\frac{a_{k}}{A_{k+1}}g_{k+1},\quad v_{k+1}=\nabla (h^\alpha)^*(-z_{k+1}).\label{def:xz_gem}
                \end{align}
                \vspace{-1.5em}
            \ENDFOR
        \end{algorithmic}
    \end{algorithm}

    The GEM is known to have the same optimal complexity as the AG method. The next result presents the complexity of our strongly convex GEM variant for computing an $\varepsilon$-certificate to~\eqref{def:probIntro}.
    \begin{theorem}\label{thm:gem_cmplx}
        
        Assume that $\max_{g\in\dom f^*}\D_{\nu^*}(g\|g_0)= D<\infty$ is bounded. Given $\varepsilon>0$ and choosing $\alpha=\varepsilon/(2M)$, then the pair $(z_k,\Gamma_k^*)$ is an $\varepsilon$-certificate for~\eqref{def:probIntro} 
        in $k=\tO(1+\sqrt{ML/\varepsilon})$ iterations of Algorithm~\ref{alg:GEM-dual}, where $\Gamma_k^*$ is the ACP model for~\eqref{eq:fenchel_rockafellar_dual} induced by $(z_0,\{z_{i+1}\}_{i=0}^{k-1}, \{a_i/A_{i+1}\}_{i=0}^{k-1})$.
    \end{theorem} 
    \noindent\textbf{Proof Sketch.} The first part of the proof of Theorem~\ref{thm:gem_cmplx} uses the smoothness of $(h^\alpha)^*$, the optimality conditions on $g_{k+1}$, and Fenchel duality to prove the single-step bound\vspace{-0.2em}
    \begin{align*}
            A_k\psi^\alpha(z_k)-\frac{\alpha a_k^2}{2\tau_{k+1}}\|v_{k+1}-v_k\|^2+\frac{\tau_k}{\alpha}\D_{\nu^*}(g\|g_k)-\frac{\tau_k}{\alpha}\D_{\nu^*}(g_{k+1}\|g_k) - \frac{\tau_{k+1}}{\alpha}\D_{\nu^*}(g\|g_{k+1})
        \\\geq A_{k+1}\psi^\alpha(z_{k+1})+a_k\left[h^\alpha(v_{k+1})+\inner{v_{k+1}}{g}-f^*(g)\right] + a_k\inner{v_{k+1}-\hat v_k}{g_{k+1}-g}.
        \end{align*}
        \vspace{-1.2em}
        
        Recursively summing this inequality from $0$ to $k-1$ and applying the strong convexity of $\nu^*$, the construction of $\hat v_k$, and duality relations yield the certificate convergence bound. 

    We finally show the most fascinating duality result of the paper: Algorithm~\ref{alg:GEM-dual} is, in fact, the dual of Algorithm~\ref{alg:three_avg_primal}, clarifying the source of TAA's acceleration. 
    
    \begin{proposition}\label{prop:3avg_duality}
        Fix $y_0\in \dom h$ and suppose that
        we choose $\nu^*=Lf^*$ in Algorithm~\ref{alg:GEM-dual}. Then, setting $s_0=z_0=g_0=\nabla f(y_0)$, $\tx_0=y_0$, Algorithm~\ref{alg:GEM-dual} solving~\eqref{eq:fenchel_rockafellar_dual} is equivalent to Algorithm~\ref{alg:three_avg_primal} solving~\eqref{def:probIntro_reg} with the following correspondence
        \footnote{Since $f^*$ is not necessarily differentiable, there is some ambiguity in the definition of $\D_{\nu^*}(g\|g_k)$ in Step \textbf{2} of Algorithm~\ref{alg:GEM-dual}. The choice used in the inductive proof of Proposition~\ref{prop:3avg_duality} is to define $\D_{f^*}(g\|g_k)$ using $\tx_k\in\partial f^*(g_k)$.}
        \begin{equation}\label{corresp:3avg}
            \nabla f(\tx_{k})=g_{k},\quad s_{k}=z_{k},\quad x_{k}=v_{k}.
        \end{equation}
    \end{proposition}

    We can once again examine the duality relation in Proposition~\ref{prop:3avg_duality} from the game-theoretic perspective of~\eqref{def:bilinear_game}. Unlike Algorithm~\ref{alg:AggGCG_primal}, we move back into sequential play with the turn structure illustrated below.
     
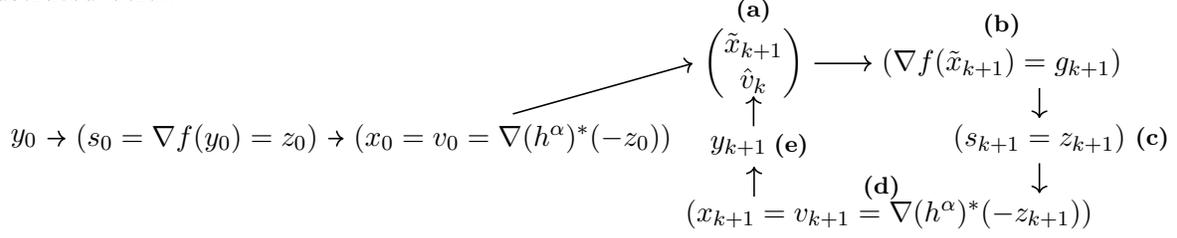
\vspace{-1.3em}\[
\begin{tikzpicture}[]
  \node (A) at (0,0) {$y_0$};

  \node[anchor=center] (B) at (2.3,0) {($s_0=\nabla f(y_0)=z_0$)};
  
  \node[anchor=center] (C) at (6.5,0) {($x_0=v_0=\nabla (h^\alpha)^*(-z_0)$)};
  
  \node[anchor=center] (D) at (9.7,1) {$\begin{pmatrix}
      \tx_{k+1}\\\hat v_{k}
  \end{pmatrix}$};

  \node[anchor=center] (F) at (13,1) {$(\nabla f(\tx_{k+1})=g_{k+1})$};
  \node[anchor=center] (G) at (13.5,0) {$(s_{k+1}=z_{k+1})$};

  \node[anchor=center] (H) at (11.5,-1.) {$(x_{k+1}=v_{k+1}=\nabla (h^\alpha)^*(-z_{k+1}))$};
  \node[anchor=center] (LD) at (9.7,1.7) {\footnotesize\rm\textbf{(a)}};
  \node[anchor=center] (LF) at (13,1.5) {\footnotesize\rm\textbf{(b)}};
  \node[anchor=center] (LG) at (15,0) {\footnotesize\rm\textbf{(c)}};
  \node[anchor=center] (LH) at (11.4,-0.65) {\footnotesize\rm\textbf{(d)}};
  \node[anchor=center] (LI) at (10.2,-0.1) {\footnotesize\rm\textbf{(e)}};
  \node[anchor=center] (I) at (9.5,-0.1) {$y_{k+1}$};
  \draw[arrows = {->[]},semithick] (A.east) -- (B.west);
  \draw[arrows = {->[]},semithick] (B.east) -- (C.west);
  \draw[arrows = {->[]},semithick] (C.north) -- (D.west);
  \draw[arrows = {->[]},semithick] (D.east) -- (F.west);
  \draw[arrows = {->[]},semithick] ($(F.south) + (5mm,-0mm)$) -- ($(F.south) + (5mm,-4mm)$);
  \draw[arrows = {->[]},semithick] ($(G.south)$) -- ($(G.south) + (0mm,-4mm)$);
  \draw[arrows = {->[]},semithick] ($(H.north west) + (10.5mm,-1mm)$) -- ($(H.north west) + (10.5mm,3mm)$);
  \draw[arrows = {->[]},semithick] ($(I.north) + (2mm,0)$) -- ($(I.north) + (2mm,4mm)$);
  
\end{tikzpicture}
\]\vspace{-2em}

The primal player moves first with the starting strategy $y_0$. The dual player responds by setting their initial strategy $z_0=\nabla f(y_0)$ and the primal player responds with $x_0=\nabla(h^\alpha)^*(-z_0)$. Then, play begins. In \textbf{(a)}, the primal player begins by computing the ``lookahead'' strategy $\tx_{k+1}$ which acts as a prediction for their next strategy ($y_{k+1}$) before seeing any dual response. Simultaneously, the dual player forms an optimistic model of the primal player response using the extrapolated point $\hat v_{k}$. In \textbf{(b)}, the primal player then predicts the dual player's move as $\nabla f(\tx_{k+1})$, while the dual player performs an optimistic update to compute the response $g_{k+1}$. 
As we show in the proof of Proposition~\ref{prop:3avg_duality}, the optimistic response $g_{k+1}$ exactly coincides with the lookahead prediction $\nabla f(\tx_{k+1})$. Next, in \textbf{(c)} the dual player updates their mixed strategy $z_{k+1}$, the primal player computes their response $x_{k+1}$ in \textbf{(d)}, and finally the primal player updates their mixed strategy $y_{k+1}$ in \textbf{(e)}. The dual player does not directly see the sequence $\{y_k\}$, however as the proof of Proposition~\ref{prop:3avg_duality} shows, the optimistic prediction $\hat v_k$ and response $g_{k+1}$ implicitly provide information about $y_k$.

In comparison with the one and two-average methods, which use zero-regret dynamics based on averaging and best-response, the three-average case leverages optimistic learning dynamics: the GEM directly from the $\hat v_{k}$ extrapolation and TAA indirectly through the underlying dual process.

\section{Connections to Fisher Market Dynamics}\label{sec:fisher}
The motivation for the algorithmic progression MDA$\to$AggGCG$\to$TAA was primarily technical: first to show convergence in the linearization sequence, then to recover optimality certificates. In this section, we show that the progression also has an intuitive appeal grounded in Fisher market dynamics. We consider a smoothed variation of the Fisher market equilibrium with linear utility functions proposed by~\cite{chenAcceleratedPriceAdjustment2025}\footnote{The original formulation in~\cite{chenAcceleratedPriceAdjustment2025} did not contain the $\alpha$ perturbation},
\begin{equation}\label{def:dual_eg}
    \phi^\alpha(\mu)=
    \overbrace{\sum_{j=1}^n\exp(\mu_j)
    +\delta\sum_{i=1}^mB_i\log\left(\sum_{j=1}^n\exp\bigl(\delta^{-1}(\log b_{ij}-\mu_j)\bigr)\right)}^{f(\mu)}
    +\overbrace{\delta_{[\underline{\mu},\overline{\mu}]^n}(\mu)}^{h(\mu)}+\overbrace{\frac{\alpha}{2}\|\mu-\mu_{\text{ref}}\|^2}^{\alpha w(\mu)},
\end{equation}
where $\mu_j$ is the logarithmic unit price of item $j$, $b_{ij}$ is buyer $i$'s valuation for item $j$, $B_i$ is buyer $i$'s budget, $\underline{\mu}\leq\overline{\mu}$ define box constraints, and $\delta>0$ is an entropic smoothing parameter. The added regularization term admits a natural interpretation as a penalization for multiplicative deviations from a benchmark price. In this example, \emph{primal variables are prices, dual variables represent demand} (incorporating individual allocations from each buyer). 


\noindent\textbf{MDA (One-Average)}. In each round, the market generates provisional prices $x_k$ by a best-fit approximation to smoothed demand $s_k$. Buyers then respond to this \textit{provisional} price in each round with the proposed demand $\nabla f(x_k)$. However, Theorem~\ref{thm:da_convergence} does not show convergence in the provisional prices, but instead the exponentially smoothed price sequence $\{y_k\}$. The market dynamics (Algorithm~\ref{alg:PDCP}) are therefore operationally asymmetric, with the market responding to a smoothed demand model $s_k$ without explicitly posting the smoothed price $y_k$.

\noindent\textbf{AggGCG (Two-Averages)}. Algorithm~\ref{alg:AggGCG_primal} remedies the asymmetry in MDA market dynamics, but loses convergence guarantees for optimality certificates. In the MDA results, the ACP model included the latest buyer allocation $\nabla f(x_{k+1})$.  As a result, there was an explicit, primal-dual relationship between the latest provisional price $x_k$ and smoothed demand forecast $s_k$. Algorithm~\ref{alg:AggGCG_primal} breaks this connection by linearizing at $y_{k-1}$ to generate $x_k$, which we can interpret as the buyers responding to stale prices $y_{k-1}$ from round $k-1$. Similarly, the market response corresponds to the smoothed demand $s_{k-1}$ from the previous round. As a result, we lose the primal-dual connections between the latest price/demand pair that admit a natural certificate. 

\noindent\textbf{TAA (Three-Averages)}. The technical motivation for a third average was primarily geometric as visualized in Figure~\ref{fig:AggGCG_taa_geometry}(b). The move to TAA also admits a more interesting market interpretation as \textit{forecasting}. Each round $k$ of Algorithm~\ref{alg:three_avg_primal} begins by predicting the next set of prices $\tx_{k+1}$. Buyers respond to the forecast price signal with $\nabla f(\tx_{k+1})$, which is then incorporated into the smoothed demand $s_{k+1}$. The market computes the best-fit prices $x_{k+1}$ for demand $s_{k+1}$, which are then added to the smoothed price $y_{k+1}$. In this dynamic, buyers do not directly respond to the convergent price sequence $\{y_k\}$, but instead to the primal-side price forecasts $\{\tx_{k+1}\}$. We can then interpret TAA as a cautious forecast model, where both primal and dual-side forecasts are hedged by averaging. Moreover, with the additional forecast sequence $\{\tx_{k+1}\}$, TAA achieves market equilibrium at faster rates than MDA and AggGCG.

\section{Conclusion}\label{sec:conclusion}
In this work, we explore primal-dual correspondences in regularized convex optimization with a focus on certifiable optimality. Motivated by primal-dual correspondences in one and two averages, we propose the TAA method which leverages three averages to obtain near-optimal complexity for solving~\eqref{def:probIntro}. While TAA is new, we show that its dual counterpart is the well-studied GEM. We show that both TAA and the GEM obtain computable $\varepsilon$-certificates for primal-dual gap convergence with $\tO(\varepsilon^{-1/2})$ complexity, a novel result for the GEM to our knowledge. Our certificate guarantees further enhance the utility for the GEM as a subsidiary solver by providing upper bounds on a computable termination condition. Furthermore, we provide in-depth intuition and motivation for our results with concrete examples in Fisher markets and zero-sum matrix games.

There are a number of directions for future work. Our analysis relies on knowledge of the problem smoothness $L$. However, $L$ may be unknown at runtime,  particularly for large-scale problem instances. Future work on universal variants of TAA would substantially improve its scope of application. Similarly, our analysis focused on \textit{smooth} problems, whereas an increasing number of methods target the class of \textit{relatively smooth} problems~\cite{luRelativelySmoothConvex2018}. Recent work~\cite{laude2023dualities} has generalized the typical smoothness/strong-convexity duality relation to the relative case, suggesting that algorithmic correspondences may exist under these generalized relations.

\appendix
\section{Technical Results}\label{appdx:technical}

    We begin by providing the proof of Lemma~\ref{lem:reg_connection}, which connects the regularized problem~\eqref{def:probIntro_reg} to the true target~\eqref{def:probIntro}.\\
    \noindent\textbf{Proof of Lemma~\ref{lem:reg_connection}}: By Fenchel-Rockafellar duality, we have
        \[
            \psi^\alpha(s)\geq \min_{z\in\R^n}\{\psi^\alpha(z)\}\stackrel{\eqref{eq:fenchel_rockafellar_dual}}=-\phi^\alpha_*,
        \]
        which implies that
        \begin{equation}\label{ineq:primal_from_pd}
            \phi^\alpha(x)-\phi_*^\alpha\leq \phi^\alpha(x)+\psi^\alpha(s).
        \end{equation}
        Let $x_*\in\underset{x\in\dom h}\Argmin \phi(x)$. Note that, by the definition of $\phi^\alpha_*$ and the boundedness of $w(x)$ on $\dom h$,
        \begin{equation}\label{ineq:alpha_conn}
            \phi^\alpha_*\leq \phi(x_*)+\alpha w(x_*)\leq \phi_*+\alpha M.
        \end{equation}
        This inequality and the non-negativity of $w$ thus yield
        \begin{equation}
            \phi(x)-\phi_*-\alpha M\stackrel{\eqref{ineq:alpha_conn}}\leq \phi(x)+\alpha w(x)-\phi^\alpha_* \stackrel{\eqref{ineq:primal_from_pd}}\leq \phi^\alpha(x)+\psi^\alpha(s) \leq \frac{\varepsilon}{2},
        \end{equation}
        where the second inequality follows from \eqref{ineq:primal_from_pd}. By the choice $\alpha\leq \varepsilon/(2M)$, we obtain
        \begin{equation}
            \phi(x)-\phi_*\leq \frac{\varepsilon}{2}+\alpha M\leq \varepsilon,
        \end{equation}
        which completes the proof.\QEDA

Next, we recall technical results relating a closed convex function $f$ to its Fenchel dual $f^*$.
    \begin{lemma}[{\cite[Theorem 4.20]{Beck2017}}]\label{lem:beck_conj_pair}
        Let $f:\R^n\to(-\infty,\infty]$ be proper, closed, and convex. Then, the following are equivalent for any $x,\,y\in\R^n$:
        \begin{itemize}
            \item[{\rm a)}] $\inner{x}{y}=f(x)+f^*(y)$,
            \item[{\rm b)}] $y\in\partial f(x)$,
            \item[{\rm c)}] $x\in\partial f^*(y)$.
        \end{itemize}
    \end{lemma}
    
    \begin{lemma}[{\cite[Corollary 4.21]{Beck2017}}]\label{lem:beck_subdiff_conj}
        Let $f:\R^n\to(-\infty,\infty]$ be proper, closed, and convex. Then, for any $x, y\in\R^n$:
        \begin{equation}
            \partial f(x)=\underset{z\in\R^n}\Argmax\{\inner{z}{x}-f^*(z)\},\quad \partial f^*(y)=\underset{u\in\R^n}\Argmax\{\inner{y}{u}-f(u)\}.
        \end{equation}
    \end{lemma}
    
    With basic duality results established, we now provide a proof of Lemma~\ref{lem:model_props}.\\
\noindent\textbf{Proof of Lemma~\ref{lem:model_props}}: Statements (a) and (b) follow by the definition of $\Gamma_k$ in~\eqref{def:ACP_model}, the convexity of $f$, and the $\alpha$-strong convexity of $h^\alpha$.

        \noindent c) Define $v_k=\argmin_{x\in\R^n}\Gamma_k(x)$. Then, by the definitions of $\Gamma_k$ and $s_k$ in~\eqref{def:ACP_model} and~\eqref{def:sk1}, respectively, we can show
        \begin{equation}
            0\in \partial \Gamma_k(v_k)=s_k+\partial h^\alpha(v_k),
        \end{equation}
        which implies that $-s_k\in\partial h^\alpha(v_k)$. By Lemma~\ref{lem:beck_conj_pair}, we obtain
        \begin{equation}\label{eq:pdcp_hdual}
            h^\alpha(v_k)+\inner{s_k}{v_k}=-(h^\alpha)^*(-s_k).
        \end{equation} 
        Similarly, we observe by Lemma~\ref{lem:beck_conj_pair} that
        \begin{equation}\label{eq:pdcp_fdual}
            f(x_{i})-\inner{\nabla f(x_{i})}{x_{i}}=-f^*(\nabla f(x_{i})),
        \end{equation}
        for all $i\in\{0,\dots,k-1\}$.
        Define the sequence of scalars $\{\chi_i\}_{i=0}^{k}$ recursively as 
        \[
        \chi_0=f(y_{0})-\inner{\nabla f(y_{0})}{y_{0}},\quad\chi_{i+1}=(1-\zeta_i)\chi_i + \zeta_i \bigl(f(x_{i})-\inner{\nabla f(x_{i})}{x_{i}}\bigr). 
        \]
        Expanding the inequality, using~\eqref{eq:pdcp_fdual} and the definition of $s_k$ in \eqref{def:sk1}, and applying the convexity of $f^*$ we obtain
        \begin{equation}\label{ineq:phi_conj}
            \chi_k \leq -f^*(s_k).
        \end{equation}
        Then, by the definitions of $v_k$ and $\Gamma_k$, we have 
        \[
        -\min_{x\in\R^n}\Gamma_k(x)=-\Gamma_k(v_k)=-\inner{s_k}{v_k}-h^\alpha(v_k)-\chi_k
        \stackrel{\eqref{eq:pdcp_hdual}\eqref{ineq:phi_conj}}\geq (h^\alpha)^*(-s_k)+f^*(s_k)=\psi^\alpha(s_k).
        \]
        Therefore, statement (c) immediately follows.\QEDA
        
    For completeness, we state the counterpart of Lemma~\ref{lem:model_props} for a dual ACP model. The proof is identical to the primal case, and is therefore omitted.
    
    \begin{lemma}\label{lem:model_props_dual}
        Let $\Gamma_k^*(\cdot)$ be the ACP model for~\eqref{eq:fenchel_rockafellar_dual} induced by $(z_0,\{g_i\}_{i=0}^{k-1},\xi)$ for $z_0\in\dom f^*$, $\{g_i\}_{i=0}^{k-1}\subseteq \dom f^*$ and $\xi\in[0,1]^k$. Then, the following statements hold:
        \begin{itemize}
            \item[{\rm a)}] for all $g\in\dom f^*$, $\Gamma_k^*(g)\leq \psi^\alpha(g)$;
            \item[{\rm b)}] $\Gamma_k^*$ is $(1/L)$-strongly convex.
        \end{itemize}
        Furthermore, define $\{v_k\}_{k\geq 0}$ as
            \[
            v_{0}=\nabla (h^\alpha)^*(-z_0),\quad v_{j+1}=(1-\xi_j)v_j+\xi_j \nabla (h^\alpha)^*(-g_{j}),
            \]
        for $0\leq j\leq k-1$. Then, the following bound holds for all $g\in \dom f^*$
        \begin{itemize}
            \item[{\rm c)}] $\phi^\alpha(v_k) + \psi^\alpha(g)\leq \psi^\alpha(g)-\min_{u\in\R^n}\Gamma_k^*(u)$.
        \end{itemize}
    \end{lemma}

    Finally, we provide a technical lemma that will be used in the analysis of Algorithm~\ref{alg:GEM-dual}.
    The statement is a slight generalization of~\cite[Lemma 3.2]{chenConvergenceAnalysisProximalLike2006} to the case of relative strong convexity and the case when $\nu$ is not necessarily differentiable. Non-differentiability does not affect our results, as our analysis of Algorithm~\ref{alg:GEM-dual} does not require any properties of the Bregman function aside from the 1-strong convexity of $\nu^*$.
    \begin{lemma}[Three-Points Inequality]\label{lem:three_point}
    Let $\Phi:\R^n\to(-\infty,\infty]$ and $\omega:\R^n\to(-\infty,\infty]$ be closed, proper, and convex functions with $\dom \Phi\subseteq\dom\omega$ and $\dom \Phi$ has a nonempty relative interior.
    Assume that $\Phi$ is $\mu$-strongly convex relative to $\omega$, that is, $\Phi(\cdot)-\mu\omega(\cdot)$ is a convex function. Fix $x_0\in \dom \Phi$, $\beta>0$ and define
    \begin{equation}
        x^+ = \underset{x\in\R^n}\argmin\left\{\Phi(x)+\beta\D_\omega(x\|x_0)\right\},
    \end{equation}
    where $\D_\omega(x\|x_0)$ is defined using an arbitrary $\omega'(x_0)\in\partial \omega(x_0)$.
    Then, for any $u\in\dom \Phi$,
    \begin{equation}\label{def:three_point}
         \Phi(x^+)+\beta\D_\omega(x^+\|x_0)+(\beta+\mu)\D_\omega(u\|x^+)\leq  \Phi(u)+\beta\D_\omega(u\|x_0),
    \end{equation}
    where $\D_\omega(\cdot\|x^+)$ is defined using some suitable $\omega'(x^+)\in\partial \omega(x^+)$.
\end{lemma}
\begin{proof}
    Define $\varphi(\cdot)= \Phi(\cdot)+\beta\D_\omega(\cdot\|x_0)$. By the optimality condition on $x^+$, we have $0\in\partial \varphi(x^+)$. Since $ \Phi(\cdot)-\mu\omega(\cdot)$ is convex, we have that $\varphi(\cdot)-(\beta+\mu)\omega(\cdot)$ is convex. By~\cite[Theorem 3.40]{Beck2017}, we have for all $x\in\dom \varphi$,
    \[
        (\beta+\mu)\partial \omega(x) + \partial [\varphi(\cdot)-(\beta+\mu)\omega(\cdot)](x)= \partial \varphi(x).
    \]
    Then, using that $0\in\partial \varphi(x^+)$ and rearranging, we have
    \begin{equation}
         -(\beta+\mu) \omega'(x^+)\in\partial (\varphi(\cdot)-(\beta+\mu)\omega(\cdot))(x^+)
    \end{equation}
    for some $\omega'(x^+)\in\partial \omega(x^+)$. It follows from the definition of the subdifferential that
    \begin{equation}
        \varphi(x^+)-(\beta+\mu)\omega(x^+)- (\beta+\mu)\inner{\omega'(x^+)}{u-x^+}\leq \varphi(u)-(\beta+\mu)\omega(u) .
    \end{equation}
    Substituting the definition of $\varphi(\cdot)$ and noting
    \begin{equation}
       \omega(u)-\omega(x^+)-\inner{\omega'(x^+)}{u-x^+}=\D_\omega(u\|x^+)
    \end{equation}
    gives the result.
\end{proof}

\section{One-Average Analysis}\label{appdx:one_avg}
We begin by providing a proof of the algorithm correspondence in Proposition~\ref{prop:1avg_duality}.\\
\noindent\textbf{Proof of Proposition~\ref{prop:1avg_duality}}: First, note that if $s_k=z_k$ for $k\geq 0$, then
        \begin{align}
            \nonumber x_{k}&=\underset{x\in\R^n}\argmin\left\{\inner{s_k}{x}+h^\alpha(x)\right\}\\
            &=\underset{x\in\R^n}\argmax\left\{\inner{-s_k}{x}-h^\alpha(x)\right\}=\nabla (h^\alpha)^*(-s_k)=\nabla (h^\alpha)^*(-z_k)\label{eq:xk_equiv}.
        \end{align}
        Similarly, by the optimality condition on $\bar z_k$, the convexity of $f$ and Lemma~\ref{lem:beck_subdiff_conj}, we have
        \begin{align}
            \bar z_k=\underset{z\in\R^n}\argmax\{\inner{\nabla (h^\alpha)^*(-z_k)}{z}-f^*(z)\}=\nabla f(\nabla (h^\alpha)^*(-z_k))\stackrel{\eqref{eq:xk_equiv}}=\nabla f(x_k).\label{eq:bar_zk_equiv}
        \end{align}
        Therefore, it is sufficient to prove that $s_k=z_k$ for all $k\geq 0$, which we will show by induction. The base case is trivial by our choice of $z_0=\nabla f(y_0)=s_0$. Now, assume $k\geq 0$ and the inductive hypothesis $s_{k}=z_k$, which implies~\eqref{eq:xk_equiv} and~\eqref{eq:bar_zk_equiv}. Then,
        \begin{equation}
            z_{k+1}=(1-\eta)z_k+\eta\bar z_k\stackrel{\eqref{eq:bar_zk_equiv}}=(1-\eta)s_k+\eta\nabla f(x_k)=s_{k+1},
        \end{equation}
        where the second equality follows from the inductive hypothesis and~\eqref{eq:bar_zk_equiv}. We thus complete the proof.~\QEDA
\subsection{Analysis of MDA}
    In this subsection, we prove the MDA certificate complexity claimed in Theorem~\ref{thm:da_convergence}.

    To begin, we define $\Gamma_k$ as the ACP model induced by $(y_0,\{x_i\}_{i=0}^{k-1},\{\eta\}^k)$. That is, $\Gamma_0(\cdot)=\ell_f(\cdot;y_0)+h^\alpha(\cdot)$ and for all $k\geq 0$, 
    \begin{equation}\label{def:gamma}
        \Gamma_{k+1}(\cdot)=(1-\eta)\Gamma_k(\cdot)+\eta\gamma_k(\cdot), \quad \gamma_k(\cdot)=\ell_f(\cdot;x_{k}) + h^{\alpha}(\cdot).
    \end{equation}
    Observe that, by construction, we have for all $x\in\dom h$
    \[
    s_{k}+\partial h^\alpha(x)=\partial \Gamma_k(x),
    \]
    hence $x_k=\argmin_{x\in\R^n}\Gamma_k(x)$ by the optimality conditions on $x_{k}$ in Algorithm~\ref{alg:PDCP}. For convenience, we define the quantities for $k\geq 0$
    \begin{equation}\label{def:mtk}
        m_k = \min_{x\in\R^n}\Gamma_{k}(x)=\Gamma_{k}(x_k),\quad t_k = \phi^\alpha(y_{k})-m_k,
    \end{equation}
    where
    \begin{equation}
    y_{k+1}=(1-\eta)y_k+\eta x_{k+1}.
    \end{equation}
    Observe that $y_k$ does not directly include the point $x_0$, however the dual average $s_k$ includes gradient information from $x_0$.

    \noindent\textbf{Proof of Theorem~\ref{thm:da_convergence}}: By the definition of $\Gamma_{k+1}$ and $m_{k+1}$, for $k\geq 0$ we obtain
        \begin{align}
            \nonumber m_{k+1}&=(1-\eta)\Gamma_{k}(x_{k+1}) + \eta\gamma_k(x_{k+1})\\
            &\geq (1-\eta)m_k + \eta\gamma_k(x_{k+1})+\frac{(1-\eta)\alpha}{2}\|x_k-x_{k+1}\|^2\tag{i}\label{proof:pdcp_sc}\\
            \nonumber &=(1-\eta)m_k + \eta\left[\gamma_k(x_{k+1})+\frac{(1-\eta)\alpha}{2\eta}\|x_k-x_{k+1}\|^2\right]\\
            &=(1-\eta)m_k + \eta\left[\gamma_k(x_{k+1})+\frac{L}{2}\|x_k-x_{k+1}\|^2\right]\tag{ii}\label{proof:pdcp_Lsub}\\
            & \geq (1-\eta)m_k + \eta\phi^\alpha(x_{k+1})\tag{iii}\label{proof:pdcp_Lsmooth},
        \end{align}
        where~\eqref{proof:pdcp_sc} follows by the 1-strong convexity of $w$,~\eqref{proof:pdcp_Lsub} by the choice $\eta=\alpha/(\alpha+L)$, and~\eqref{proof:pdcp_Lsmooth} by the $L$-smoothness of $f$. Then, expanding the inequality from $1$ to $k+1$, we have
        \begin{align*}
            m_{k+1}&\geq (1-\eta)^{k+1}m_0+\eta\sum_{i=1}^{k+1}(1-\eta)^{k+1-i}\phi^\alpha(x_{i})\\
            &=-(1-\eta)^{k+1}t_0+(1-\eta)^{k+1}\phi^\alpha(y_0)+\eta\sum_{i=1}^{k+1}(1-\eta)^{k+1-i}\phi^\alpha(x_{i})\\
            &\geq -(1-\eta)^{k+1}t_0 + \phi^\alpha(y_{k+1}),
        \end{align*}
        where the second line follows by the definition of $t_0$, and the third by the convexity and the definition of $y_{k+1}$ and induction.
        Therefore, we obtain
        \begin{equation}
            t_{k+1} = \phi^\alpha(y_{k+1})-m_{k+1} \leq (1-\eta)^{k+1}t_0.
        \end{equation}
       Since $1-\eta=L/(\alpha+L)=(1+\alpha/L)^{-1}$, we have for all $k\geq 0$
        \[
            \phi^\alpha(y_{k})-\min_{x\in\R^n}\Gamma_k(x)\leq \frac{t_0}{\left(1+\frac{\alpha}{L}\right)^{k}}.
        \]
        As a result, we obtain a point $y_k$ satisfying $\phi^\alpha(y_{k})-\min_{x\in\dom h}\Gamma_k(x)\leq \varepsilon/2$ after
        \begin{equation}
            k=\mathcal{O}\left(1+\frac{L}{\alpha}\log \left(\frac{t_0}{\varepsilon}\right)\right)
        \end{equation}
        iterations.
        Choosing $\alpha=\varepsilon/(2M)$ with Definition~\ref{def:pd_cert} yields the complexity result.\QEDA
        
    \subsection{Analysis of GCG}
    In this subsection, we prove the complexity bound stated in Theorem~\ref{thm:CG-cmplx}.
    
    We begin by defining the Wolfe gap for \eqref{eq:fenchel_rockafellar_dual} as
    \begin{equation}\label{def:wolfe_gap}
        S(z)=\max_{v\in\R^n}\left\{\inner{-\nabla (h^\alpha)^*(-z)}{z-v}+f^*(z)-f^*(v)\right\}.
    \end{equation}
    
    The following lemma provides a useful relation between the Wolfe gap and the primal-dual gap.
    \begin{lemma}\label{lem:wolfe_gap_pd}
        For $z\in\dom f^*$, define $y=\nabla (h^\alpha)^*(-z)$. Then, we have
        \begin{equation}
            S(z)=\psi^\alpha(z)+\phi^\alpha(y).
        \end{equation}
    \end{lemma}
    \begin{proof}
        Using \eqref{def:wolfe_gap} and the definition of $y$, we have
        \begin{align*}
            S(z) &\stackrel{\eqref{def:wolfe_gap}}= \max_{v\in\R^n}\{\inner{-y}{z-v}+f^*(z)-f^*(v)\}\\
            &=f^*(z)+\inner{y}{-z}+\max_{v\in\R^n}\{\inner{y}{v}-f^*(v)\}\\
            &=f^*(z)+h^\alpha(y)+(h^\alpha)^*(-z)+f(y)\\
            &\stackrel{\eqref{def:probIntro_reg}, \eqref{eq:fenchel_rockafellar_dual}}=\psi^\alpha(z)+\phi^\alpha(y),
        \end{align*}
        where the third line follows by Lemma~\ref{lem:beck_conj_pair}(a) and the fourth line follows by the definitions of $\phi^\alpha$ and $\psi^\alpha$ in \eqref{def:probIntro_reg} and \eqref{eq:fenchel_rockafellar_dual}, respectively.
    \end{proof}

    The Wolfe gap therefore acts as a primal-dual certificate of optimality, as has been shown in numerous prior works~\cite{liangPrimaldualProximalBundle2025a,jaggiRevisitingFrank2013}. Accordingly, we can explicitly connect the Wolfe gap~\eqref{def:wolfe_gap} to the notion of an SCP model, as the following lemma shows.
    For notational convenience, we define the linearization $\ell_{(h^\alpha)^*}(\cdot;z)$ as
    \begin{equation}\label{def:ell}
        \ell_{(h^\alpha)^*}(\cdot;z)=(h^\alpha)^*(-z)+\inner{-\nabla (h^\alpha)^*(-z)}{\cdot-z}.
    \end{equation}
    \begin{lemma}\label{lem:wolfe_gap_to_acp}
        Fix $z\in\dom f^*$. Then, we have
        \begin{equation}\label{eq:wolfe_gap_to_acp}
            S(z)=\psi^\alpha(z)-\min_{v\in\R^n}\Gamma^*(v),
        \end{equation}
        where $\Gamma^*(\cdot)=\ell_{(h^\alpha)^*}(\cdot;z)+f^*(\cdot)$ is the SCP model of~\eqref{eq:fenchel_rockafellar_dual} induced by $z$.
    \end{lemma}
    \begin{proof}
        By the SCP construction in Definition~\ref{def:acp_model} with $k=0$, $\Gamma^*$ has the simplified, one-cut form
        \begin{equation}\label{def:Gamma_star_z}
            \Gamma^*(\cdot)=\ell_{(h^\alpha)^*}(\cdot;z)+f^*(\cdot).
        \end{equation}
        Then,
        \begin{align*}
            \psi^\alpha(z)-\min_{v\in\R^n} \Gamma^*(v)&\stackrel{\eqref{def:Gamma_star_z}}=\psi^\alpha(z)-\min_{v\in\R^n} \{(h^\alpha)^*(-z)+\inner{-\nabla (h^\alpha)^*(-z)}{v-z}+f^*(v)\}\\
            &=f^*(z)-\min_{v\in\R^n} \{\inner{-\nabla (h^\alpha)^*(-z)}{v-z}+f^*(v)\}\\
            &=\max_{v\in\R^n} \{\inner{-\nabla (h^\alpha)^*(-z)}{z-v}+f^*(z)-f^*(v)\}\stackrel{\eqref{def:wolfe_gap}}=S(z),
        \end{align*}
        where the first equality follows by the definition of $\Gamma^*$ and the final equality by the definition of $S(z)$, which concludes the proof.
    \end{proof}
    The following lemma is adapted from~{\cite[Lemma 13.7]{Beck2017}} to the case where $f^*$ is $L^{-1}$-strongly convex with our choice of $\eta$.
     \begin{lemma}\label{lem:tau_cancelling_gcg}
        Choosing $\eta=\alpha/(L+\alpha)$, we have for all iterations $k\geq0$
        \begin{equation}
            \psi^\alpha(z_{k+1})\leq \psi^\alpha(z_k)-\eta S(z_k).
        \end{equation}
    \end{lemma}
    \begin{proof}
        It follows from the $\alpha^{-1}$-smoothness of $(h^\alpha)^*(-\cdot)$ and the definition of $z_{k+1}$ in \eqref{def:z+} that
        \begin{align}
            (h^\alpha)^*(-z_{k+1})\leq& (h^\alpha)^*(-z_k)+\inner{-\nabla (h^\alpha)^*(-z_k)}{z_{k+1}-z_k} + \frac{1}{2\alpha}\|z_{k+1}-z_k\|_*^2\\
            \stackrel{\eqref{def:z+}}=&(h^\alpha)^*(-z_k)+\eta\inner{-\nabla (h^\alpha)^*(-z_k)}{ \bar{z}_k-z_k}+\frac{\eta^2}{2\alpha}\|\bar{z}_k-z_k\|_*^2.\label{ineq:gcg_ha_bound}
        \end{align}
        Similarly, by the $L^{-1}$-strong convexity of $f^*$, we have
        \begin{equation}
            f^*(z_{k+1})\leq f^*(z_k)+\eta(f^*(\bar{z}_k)-f^*(z_k)) -\frac{\eta(1-\eta)}{2L}\|z_k-\bar{z}_k\|_*^2.\label{ineq:gcg_fstar_bound}
        \end{equation}
        Combining the above bounds, we have
        \begin{align}
            \nonumber\psi^\alpha(z_{k+1})\stackrel{\eqref{eq:fenchel_rockafellar_dual}}=&(h^\alpha)^*(-z_{k+1})+f^*(z_{k+1})\\
            \stackrel{\eqref{ineq:gcg_ha_bound}\eqref{ineq:gcg_fstar_bound}}\leq& (h^\alpha)^*(-z_k)+\eta\inner{-\nabla (h^\alpha)^*(-z_k)}{ \bar{z}_k-z_k}+\frac{\eta^2}{2\alpha}\|\bar{z}_k-z_k\|_*^2\\
            \nonumber&+f^*(z_k)+\eta(f^*(\bar{z}_k)-f^*(z_k)) -\frac{\eta(1-\eta)}{2L}\|z_k-\bar{z}_k\|_*^2\\
            \nonumber=& \psi^\alpha(z_k) -\eta[\inner{-\nabla (h^\alpha)^*(-z_k)}{ z_k-\bar{z}_k}+f^*(z_k)-f^*(\bar{z}_k)]\\
            \nonumber&+\frac{1}{2}\left(\alpha^{-1}\eta^2-L^{-1} \eta(1-\eta)\right)\|z_k-\bar{z}_k\|_*^2\\
            =&\psi^\alpha(z_k) -\eta[\inner{-\nabla (h^\alpha)^*(-z_k)}{ z_k-\bar{z}_k}+f^*(z_k)-f^*(\bar{z}_k)],\label{ineq:zkp1_ineq_psi}
        \end{align}
        where the last equality follows from
        the choice $\eta=\alpha/(L+\alpha)$ and
        \[
            \alpha^{-1}\eta^2-L^{-1}\eta(1-\eta)=\frac{\alpha}{(L+\alpha)^2}-\frac{\alpha}{(L+\alpha)^2}= 0.
        \]
        Using \eqref{def:barz_gcg} and \eqref{def:wolfe_gap}, we have
        \[
        S(z_k)=\inner{-\nabla (h^\alpha)^*(-z_k)}{z_k-\bar{z}_k}+f^*(z_k)-f^*(\bar{z}_k),
        \]
        which together with \eqref{ineq:zkp1_ineq_psi} proves the claim of the lemma.
    \end{proof}
    Then, we have the following convergence guarantee.
    
    \begin{proposition}\label{prop:CG-converge} Define $\tilde y_0 = \nabla (h^{\alpha})^*(-z_0)$, and for all $k\geq 0$, 
    \begin{equation}\label{def:gcg_y}
    \tilde y_{k+1}=(1-\eta) \tilde y_{k}+ \eta \nabla (h^{\alpha})^*(-z_k).
    \end{equation}
    Then, for all $k\geq 0$, the primal-dual pair $(\tilde y_{k},z_{k})$ satisfies
    \begin{equation}
        \psi^\alpha(z_{k})+\phi^\alpha(\tilde y_{k})\leq\frac{\psi^\alpha(z_0)+\phi^\alpha(\tilde y_0)}{\left(1+\frac{\alpha}{L}\right)^{k}}.
    \end{equation}
    \end{proposition}
    \begin{proof}
    By Lemma~\ref{lem:wolfe_gap_pd}, for $k\geq 0$ we have
    \begin{equation}\label{ineq:tau_gap_expand}
        -\eta S(z_k)=-\eta\psi^\alpha(z_k) -\eta\phi^\alpha(\nabla (h^\alpha)^*(-z_k))
        \stackrel{\eqref{def:gcg_y}}\leq -\phi^\alpha(\tilde y_{k+1})+(1-\eta)\phi^\alpha(\tilde y_k)-\eta\psi^\alpha(z_k),
    \end{equation}
    where the inequality follows by the convexity of $\phi^\alpha$ and the definition of $\tilde y_{k+1}$ in~\eqref{def:gcg_y}.
    Applying Lemma~\ref{lem:tau_cancelling_gcg}, we obtain
    \begin{equation}
        \psi^\alpha(z_{k+1})\leq \psi^\alpha(z_k)-\eta S(z_k)\stackrel{\eqref{ineq:tau_gap_expand}}\leq \psi^\alpha(z_k)-\phi^\alpha(\tilde y_{k+1})+(1-\eta)\phi^\alpha(\tilde y_k)-\eta\psi^\alpha(z_k),
    \end{equation}
    which implies
    \begin{equation}
        \psi^\alpha(z_{k+1})+\phi^\alpha(\tilde y_{k+1})\leq (1-\eta)[\psi^\alpha(z_k)+\phi^\alpha(\tilde y_k)].
    \end{equation}
    Recursively expanding the inequality from $0$ to $k-1$ yields
    \begin{equation}
        \psi^\alpha(z_k)+\phi^\alpha(\tilde y_k)\leq(1-\eta)^{k}[\psi^\alpha(z_0)+\phi^\alpha(\tilde y_0)].
    \end{equation}
    Using $(1-\eta)=L/(\alpha+L)=(1+\alpha/L)^{-1}$ gives the claimed convergence bound.
    \end{proof}
    We are now ready to prove Theorem~\ref{thm:CG-cmplx}.

    \noindent\textbf{Proof of Theorem~\ref{thm:CG-cmplx}}: First, note that by~\eqref{eq:fenchel_rockafellar_dual} and Proposition~\ref{prop:CG-converge}, we have for all $k\geq 0$
        \begin{equation}\label{ineq:fw_primal_gap}
            \psi^\alpha(z_{k})-\psi^\alpha_*\leq \psi^\alpha(z_{k})+\phi^\alpha(\tilde{y}_{k})\leq \frac{\psi^\alpha(z_{0})+\phi^\alpha(\tilde{y}_{0})}{(1+\frac{\alpha}{L})^{k}}.
        \end{equation}
        Then, by Lemma~\ref{lem:tau_cancelling_gcg}, we have
        \begin{equation}
            0\leq \psi^\alpha(z_{k+1})-\psi^\alpha_*\leq \psi^\alpha(z_{k})-\psi^\alpha_*-\eta S(z_k)\stackrel{\eqref{eq:wolfe_gap_to_acp}}=\psi^\alpha(z_{k})-\psi^\alpha_*-\eta \left(\psi^\alpha(z_{k})-\min_{v\in\R^n}\Gamma_k^*(v)\right).
        \end{equation}
        where the equality follows by Lemma~\ref{lem:wolfe_gap_to_acp}. Hence, we obtain
        \begin{equation}
            \psi^\alpha(z_{k})-\min_{v\in\R^n}\Gamma_k^*(v)\leq \frac{\psi^\alpha(z_{k})-\psi^\alpha_*}{\eta} \stackrel{\eqref{ineq:fw_primal_gap}}\leq\frac{\psi^\alpha(z_{0})+\phi^\alpha(\tilde{y}_{0})}{\eta(1+\frac{\alpha}{L})^{k}}.
        \end{equation}   
        By standard analysis and our choice of $\eta$, the complexity to obtain $\psi^\alpha(z_{k})-\min_{v\in\R^n}\Gamma_k^*(v)\leq\varepsilon/2$ is therefore
        \[
        k=\mathcal{O}\left(1+\frac{L}{\alpha}\log\left(\frac{(L+\alpha)(\psi^\alpha(z_{0})+\phi^\alpha(\tilde{y}_{0}))}{\alpha\varepsilon}\right)\right).
        \]
        The result follows by defining the dual certificate (analogous to Definition~\ref{def:pd_cert}) as the pair $(z_k,\Gamma_k^*)$
        and the choice $\alpha=\varepsilon/(2M)$.\QEDA
        
        \section{Two-Average Analysis}\label{appdx:two_avg}
     We begin by providing a formal proof of the self-duality claimed in Proposition~\ref{prop:self_dual}.\\
     \noindent\textbf{Proof of Proposition~\ref{prop:self_dual}}: We prove that $y_k=v_k$ and $s_k=z_k$ for all $k\geq 0$ by induction. The base case $k=0$ follows by our initialization $s_0=z_0$ and $y_0=v_0$. Then, assume that for some $k\geq 0$, the equalities $s_{k}=z_{k}$ and $y_{k}=v_{k}$ hold. First, observe that by Lemma~\ref{lem:beck_subdiff_conj} and the smoothness of $(h^\alpha)^*$,
        \begin{equation}\label{eq:tx_equiv_2}
            x_{k+1}=\underset{x\in\R^n}\argmax\left\{-\inner{s_{k}}{x}-h(x)-\alpha w(x)\right\}=\nabla (h^\alpha)^*(-s_{k})=\nabla (h^\alpha)^*(-z_{k}).
        \end{equation}
       Then, by the choice of $v_{k+1}$ in Step \textbf{2}  of Algorithm~\ref{alg:AggGCG_dual}, we have
        \begin{equation}\label{eq:x_k_equiv}
            y_{k+1}=\eta x_{k+1}+(1-\eta)y_{k}\stackrel{\eqref{eq:tx_equiv_2}}=\eta\nabla (h^\alpha)^*(-z_{k})+(1-\eta)v_{k}=v_{k+1},
        \end{equation}
        where the second equality follows by~\eqref{eq:tx_equiv_2} and the inductive hypothesis. Then, we have
        \begin{equation}\label{eq:bar_zk_equiv_2}
            \bar z_{k+1}=\argmin\{\inner{-v_{k}}{z}+f^*(z)\}\stackrel{\eqref{eq:x_k_equiv}}=\argmax\{\inner{y_{k}}{z}-f^*(z)\}=\nabla f(y_{k}).
        \end{equation}
        Finally, the dual average satisfies
        \[
            z_{k+1}=\eta\bar z_{k+1} + (1-\eta)z_{k}\stackrel{\eqref{eq:bar_zk_equiv_2}}=\eta\nabla f(y_{k})+(1-\eta)z_{k}
            =\eta\nabla f(y_{k})+(1-\eta)s_{k}=s_{k+1},
        \]
        completing the inductive step.  We therefore conclude that the equivalence $(y_k,s_k)=(v_k,z_k)$ holds for all iterations $k\geq 0$. Then,~\eqref{eq:tx_equiv_2} and~\eqref{eq:bar_zk_equiv_2} imply $(x_{k+1},\nabla f(y_k))=(\nabla (h^\alpha)^*(-z_k),\bar z_{k+1})$ for all $k\geq 0$, concluding the proof.\QEDA

    With the correspondence proven, the rest of the section is devoted to proving the primal-dual gap convergence rates claimed in Theorem~\ref{thm:agg_GCG_conv}. Our argument and results are similar to previous work~\cite{zhao2023generalized}, hence we include the proofs primarily for completeness. For simplicity, we adopt the primal perspective of Algorithm~\ref{alg:AggGCG_primal}.
 \begin{lemma}\label{lem:gcg_per_step_duality}
        For all iterations $k\geq 0$ of Algorithm~\ref{alg:AggGCG_primal}, the following inequalities hold:
        \begin{itemize}
            \item[{\rm a)}]
            \begin{equation}
                f(y_{k+1})\leq (1-\eta)f(y_{k})-\eta f^*(\nabla f(y_{k}))+\eta \inner{\nabla f(y_{k})}{x_{k+1}}+\frac{L\eta ^2}{2}\|x_{k+1}-y_{k}\|^2;
            \end{equation}
            \item[{\rm b)}]
            \begin{equation}
                (h^\alpha)^*(-s_{k+1})\leq (1-\eta )(h^\alpha)^*(-s_{k})-\eta h^\alpha(x_{k+1})-\eta \inner{x_{k+1}}{\nabla f(y_{k})}+\frac{\eta ^2}{2\alpha}\|\nabla f(y_{k})-s_{k}\|_*^2;
            \end{equation}
            \item[{\rm c)}]
            \begin{equation}
                h^\alpha(y_{k+1})\leq \eta  h^\alpha(x_{k+1})+(1-\eta )h^\alpha(y_{k})-\frac{\alpha\eta (1-\eta )}{2}\|x_{k+1}-y_k\|^2;
            \end{equation}
            \item[{\rm d)}]
            \begin{equation}
                f^*(s_{k+1})\leq \eta  f^*(\nabla f(y_{k}))+(1-\eta )f^*(s_{k})-\frac{\eta (1-\eta )}{2L}\|\nabla f(y_{k})-s_{k}\|_*^2.
            \end{equation}
        \end{itemize}
    \end{lemma}
    \begin{proof}
       a) Since $f$ is $L$-smooth, we have
    \begin{align*}
        f(y_{k+1})\leq& f(y_{k}) + \inner{\nabla f(y_{k})}{y_{k+1}-y_{k}}+\frac{L}{2}\|y_{k+1}-y_{k}\|^2\\
        \stackrel{\eqref{def:y_agg_gcg}}=& f(y_{k}) + \eta \inner{\nabla f(y_{k})}{x_{k+1}-y_{k}}+\frac{L\eta ^2}{2}\|x_{k+1}-y_{k}\|^2\\
        =&f(y_{k}) -\eta \inner{\nabla f(y_{k})}{y_{k}}+ \eta \inner{\nabla f(y_{k})}{x_{k+1}}+\frac{L\eta ^2}{2}\|x_{k+1}-y_{k}\|^2\\
        =&(1-\eta )f(y_{k}) -\eta f^*(\nabla f(y_{k}))+ \eta \inner{\nabla f(y_{k})}{x_{k+1}}+\frac{L\eta ^2}{2}\|x_{k+1}-y_{k}\|^2,
    \end{align*}
    where the last line follows by Lemma~\ref{lem:beck_conj_pair} with $x=y_{k}$, $y=\nabla f(y_{k})$.

    \noindent b) Since $(h^\alpha)^*(-\cdot)$ is $\alpha^{-1}$-smooth, we can follow the same steps as in (a) to obtain
    \begin{align*}
       (h^\alpha)^*(-s_{k+1})\leq& (h^\alpha)^*(-s_{k})+\inner{-\nabla (h^\alpha)^*(-s_{k})}{s_{k+1}-s_{k}}+\frac{1}{2\alpha}\|s_{k+1}-s_{k}\|_*^2\\
       \stackrel{\eqref{def:s_agg_gcg}}=&(h^\alpha)^*(-s_{k})+\eta \inner{-\nabla (h^\alpha)^*(-s_{k})}{\nabla f(y_{k})-s_{k}}+\frac{\eta ^2}{2\alpha}\|\nabla f(y_{k})-s_{k}\|_*^2\\
       =&(h^\alpha)^*(-s_{k})+\eta \inner{\nabla (h^\alpha)^*(-s_{k})}{s_{k}}-\eta \inner{\nabla (h^\alpha)^*(-s_{k})}{\nabla f(y_{k})}+\frac{\eta ^2}{2\alpha}\|\nabla f(y_{k})-s_{k}\|_*^2\\
       =&(1-\eta )(h^\alpha)^*(-s_{k})-\eta  h^\alpha(x_{k+1})-\eta \inner{x_{k+1}}{\nabla f(y_{k})}+\frac{\eta ^2}{2\alpha}\|\nabla f(y_{k})-s_{k}\|_*^2,
    \end{align*}
    where the last line follows by Lemma~\ref{lem:beck_conj_pair} with $x=x_{k+1}$, $y=s_{k}$.

    \noindent Statement c) follows directly from the $\alpha$-strong convexity of $h^\alpha$ with respect to $\|\cdot\|$ and~\eqref{def:y_agg_gcg}, and statement d) follows from the $L^{-1}$-strong convexity of $f^*$ with respect to $\|\cdot\|_*$ and~\eqref{def:s_agg_gcg}.
    \end{proof}

    Using the smoothness/convexity bounds in Lemma~\ref{lem:gcg_per_step_duality}, we are now ready to prove Theorem~\ref{thm:agg_GCG_conv}.
    \noindent\textbf{Proof of Theorem~\ref{thm:agg_GCG_conv}}: Using the choice $\eta=\alpha/(L+\alpha)$ (hence $1-\eta=L/(L+\alpha)$), we obtain
    \begin{equation}
        L\eta^2-\eta(1-\eta)\alpha=\frac{L\alpha^2}{(L+\alpha)^2}-\frac{L\alpha^2}{(L+\alpha)^2}= 0,\label{eq:eta_cancel_primal}
    \end{equation}
    and therefore
    \begin{equation}
        \frac{\eta^2}{\alpha}-\frac{\eta(1-\eta)}{L}=\frac{1}{L\alpha}\left(L\eta^2-\eta(1-\eta)\alpha\right)\stackrel{\eqref{eq:eta_cancel_primal}}= 0.\label{eq:eta_cancel_dual}
    \end{equation}
    Then, from Lemma~\ref{lem:gcg_per_step_duality}(a) and (c), we have
        \begin{align}
            \nonumber \phi^\alpha(y_{k+1})=&f(y_{k+1})+h^\alpha(y_{k+1})\leq (1-\eta)\phi^\alpha(y_{k})-\eta f^*(\nabla f(y_{k}))+\eta h^\alpha(x_{k+1}) + \eta\inner{\nabla f(y_{k})}{x_{k+1}}\\
            \nonumber  &+\frac{1}{2}(L\eta^2-\alpha(1-\eta)\eta)\|x_{k+1}-y_{k}\|^2\\
            \stackrel{\eqref{eq:eta_cancel_primal}}=&(1-\eta)\phi^\alpha(y_{k})-\eta f^*(\nabla f(y_{k}))+\eta h^\alpha(x_{k+1}) + \eta\inner{\nabla f(y_{k})}{x_{k+1}}\label{ineq:primal_ineq_gcg},
        \end{align}
        and by Lemma~\ref{lem:gcg_per_step_duality}(b) and (d), we have
        \begin{align}
            \nonumber \psi^\alpha(s_{k+1})=&f^*(s_{k+1})+(h^\alpha)^*(-s_{k+1})\leq (1-\eta)\psi^\alpha(s_k)-\eta h^\alpha(x_{k+1})\\
           \nonumber  &+\eta f^*(\nabla f(y_{k})) - \eta\inner{x_{k+1}}{\nabla f(y_{k})}
            +\frac{1}{2}(\alpha^{-1}\eta^2-L^{-1}(1-\eta)\eta )\|\nabla f(y_{k})-s_k\|_*^2\\
            \stackrel{\eqref{eq:eta_cancel_dual}}=&(1-\eta)\psi^\alpha(s_k)-\eta h^\alpha(x_{k+1})+\eta f^*(\nabla f(y_{k})) - \eta\inner{x_{k+1}}{\nabla f(y_{k})}\label{ineq:dual_ineq_gcg}.
        \end{align}
        Summing~\eqref{ineq:primal_ineq_gcg} and~\eqref{ineq:dual_ineq_gcg} and canceling terms gives
        \[\phi^\alpha(y_{k+1})+\psi^\alpha(s_{k+1})\leq (1-\eta)[\phi^\alpha(y_{k})+\psi^\alpha(s_k)]\leq (1-\eta)^{k+1}[\phi^\alpha(y_{0})+\psi^\alpha(s_{0})].\]
        By standard analysis and our choice of $\eta$, the complexity to obtain $\phi^\alpha(y_{k})+\psi^\alpha(s_{k})\leq\varepsilon/2$ is therefore
        \[
        k=\mathcal{O}\left(1+\frac{L}{\alpha}\log\left(\frac{\phi^\alpha(y_{0})+\psi^\alpha(s_{0})}{\varepsilon}\right)\right),
        \] 
        which gives the first complexity bound in view of $\alpha=\varepsilon/(2M)$. The second complexity bound follows directly from Lemma~\ref{lem:reg_connection}.\QEDA
    
    \section{Three-Average Analysis}\label{appdx:three_avg}
    As in the previous two sections, we begin by proving the claimed primal-dual correspondence stated in Proposition~\ref{prop:3avg_duality}.
    
    \noindent\textbf{Proof of Proposition~\ref{prop:3avg_duality}}: 
        First, we observe that Lemma~\ref{lem:gem_technical}(a) and~\eqref{eq:gem_scalars} imply
        \[
        a_k=A_k\left(\frac{\alpha+\sqrt{\alpha^2 + 4\alpha L}}{2L}\right).
        \]
        Then, we obtain
        \begin{align*}
            \frac{a_k}{A_{k+1}}=\frac{\alpha+\sqrt{\alpha^2+4\alpha L}}{2L+\alpha +\sqrt{\alpha^2+4\alpha L}}=\frac{2\alpha}{\alpha + \sqrt{\alpha^2+4\alpha L}}\stackrel{\eqref{eq:x0_and_lambda}}=\lambda.
        \end{align*}
        It thus follows from Algorithm \ref{alg:three_avg_primal} that for all $k\geq0$,
        \begin{gather}
        \tx_{k+1}=\frac{A_k}{A_{k+1}}y_k + \frac{a_k}{A_{k+1}}x_{k},\quad y_{k+1}=\frac{A_k}{A_{k+1}}y_k + \frac{a_k}{A_{k+1}}x_{k+1};\label{eq:tx_y_equiv}\\
        \nonumber s_{k+1}=\frac{A_k}{A_{k+1}}s_k + \frac{a_k}{A_{k+1}}\nabla f(\tx_{k+1}).
        \end{gather}

        We now prove the correspondence $(g_{k},z_{k},v_k)=(\nabla f(\tx_k),s_k,x_k)$ holds for all iterations $k\geq 0$ by induction. The base case follows by our choice of initialization with $g_0=\nabla f(y_0)=\nabla f(\tx_0)$ and by Lemma~\ref{lem:beck_subdiff_conj} applied to $x_0$ as defined in~\eqref{eq:x0_and_lambda}. Now, assume that $(g_{k},z_{k},v_k)=(\nabla f(\tx_k),s_k,x_k)$ for some $k\geq0$ and define $x_{-1}=x_0$.

        Note that the subdifferential $\partial f^*(g_k)$ is not necessarily a singleton, therefore there is some ambiguity in defining $\D_{\nu^*}(g\|g_k)$ as used in Step \textbf{2} of Algorithm~\ref{alg:GEM-dual}. Since $\tx_k\in\partial f^*(g_k)$ by the inductive hypothesis and Lemma~\ref{lem:beck_conj_pair}, we choose $(f^*)'(g_k)=\tx_k$ for the linearization term, so
        \[
        \D_{\nu^*}(g\|g_k)=L\D_{f^*}(g\|g_k)=L\left[f^*(g)-f^*(g_k)-\inner{(f^*)'(g_k)}{g-g_k}\right].
        \]
        
        Then, by the optimality condition of~\eqref{def:g_gem} with $\nu^*=Lf^*$, for some $(f^*)'(g_{k+1})\in\partial f^*(g_{k+1})$ we have
        \begin{equation*}
            0=\frac{L}{\alpha}\left(\tau_k+\frac{\alpha}{L} a_k\right)(f^*)'(g_{k+1})- \frac{L\tau_k}{\alpha} (f^*)'(g_k)-a_k\hat v_k.
        \end{equation*}
        Now, we have $\tau_k + \alpha a_kL^{-1}=\tau_{k+1}=\alpha L^{-1}A_{k+1}$. Then, rearranging, we obtain
        \begin{align*}
            A_{k+1}(f^*)'(g_{k+1})&= L\alpha^{-1}\tau_{k+1}(f^*)'(g_{k+1})=L\alpha^{-1}\tau_k(f^*)'(g_k)+a_k\hat v_k\\
            &\stackrel{\eqref{eq:gem_hatx}}=L\alpha^{-1}\tau_k(f^*)'(g_k)+a_k v_k + a_{k-1}(v_k-v_{k-1})\\
            &=A_k\tx_k+a_k x_k + a_{k-1}(x_k-x_{k-1})\\
            &\stackrel{\eqref{eq:tx_y_equiv}}=A_{k}y_{k}+a_k x_k \stackrel{\eqref{eq:tx_y_equiv}}=A_{k+1}\tx_{k+1},
        \end{align*}
        where the second line follows from~\eqref{eq:gem_hatx}, the third line from the inductive hypothesis, and the fourth line from~\eqref{eq:tx_y_equiv}. Therefore $\tx_{k+1}\in\partial f^*(g_{k+1})$, hence by Lemma~\ref{lem:beck_conj_pair} we obtain
\begin{equation}\label{eq:tx_g_corresp}
            \nabla f(\tx_{k+1})=g_{k+1}.
        \end{equation}
        With this correspondence, we note that 
        \begin{equation}\label{eq:s_z_corresp_3avg}
            s_{k+1}=\frac{A_{k}}{A_{k+1}}s_{k}+\frac{a_{k}}{A_{k+1}}\nabla f(\tx_{k+1})\stackrel{\eqref{eq:tx_g_corresp}}=\frac{A_{k}}{A_{k+1}}z_{k}+\frac{a_{k}}{A_{k+1}}g_{k+1}=z_{k+1},
        \end{equation}
        where the equality follows by the inductive hypothesis and~\eqref{eq:tx_g_corresp}. Finally,~\eqref{def:x_3avg} and Lemma~\ref{lem:beck_subdiff_conj} imply that
        \begin{equation}
            x_{k+1}=\nabla (h^\alpha)^*(-s_{k+1})=\nabla (h^\alpha)^*(-z_{k+1})=v_{k+1}.
        \end{equation}
        We therefore have $(g_{k+1},z_{k+1},v_{k+1})=(\nabla f(\tx_{k+1}), s_{k+1}, x_{k+1})$, completing the proof.\QEDA

\subsection{Analysis of TAA}
Throughout this subsection, we define $\Gamma_{k}(\cdot)$ as the ACP model induced by $(y_0, \{\tx_{i+1}\}_{i=0}^{k-1},\{\lambda\}^k)$. Recall that this implies the recursive definition
    \begin{equation}\label{def:gamma_3avg}
        \Gamma_{k+1}(\cdot)=(1-\lambda)\Gamma_k(\cdot)+\lambda\gamma_k(\cdot), \quad \gamma_k(\cdot)=\ell_f(\cdot;\tx_{k+1}) + h^{\alpha}(\cdot),
    \end{equation}
    with $\Gamma_0(\cdot)=\ell_f(\cdot;y_0)+h^\alpha(\cdot)$. By simple induction and the definition of $s_k$ in~\eqref{def:s_3avg}, we can show that $x_{k}=\argmin_{x\in\R^n}\Gamma_{k}(\cdot)$ for all $k\geq 0$. Then, we have the following proposition.

    \begin{proposition}\label{prop:single_step_3avg}
        For all $k\geq 0$, the following inequality holds
        \begin{equation}
            \min_{x\in\R^n}\Gamma_k(x)\geq\phi^\alpha(y_k) -(1-\lambda)^{k}\Delta,
        \end{equation}
        where $\Delta=\phi^\alpha(y_0)-\Gamma_0(x_0)$.
    \end{proposition}
    \begin{proof}
    First, note that by our choice of $\lambda$ and simple algebra, we can show the relationship 
    \begin{equation}\label{eq:lambda_alpha}
        \frac{L}{\alpha}= \frac{1-\lambda}{\lambda^{2}}.
    \end{equation}
        We now prove the claim by induction on $k$. The base case trivially holds, since
        \begin{equation}
            \phi^\alpha(y_0)-(1-\lambda)^0\Delta=\Gamma_0(x_0)\stackrel{\eqref{eq:x0_and_lambda}}=\min_{x\in\R^n}\Gamma_0(x).
        \end{equation}
        Now suppose that the claim holds for some $k\geq 0$. Then, since $x_{k+1}=\argmin_{x\in\R^n}\Gamma_{k+1}(x)$ by construction, we have
        \begin{align*}
            \min_{x\in\R^n}\Gamma_{k+1}(x)&=\Gamma_{k+1}(x_{k+1})\stackrel{\eqref{def:gamma_3avg}}=(1-\lambda)\Gamma_k(x_{k+1}) + \lambda\gamma_k(x_{k+1})\\
            &\geq (1-\lambda)\min_{x\in\R^n}\Gamma_{k}(x) + \lambda\gamma_k(x_{k+1}) + \frac{(1-\lambda)\alpha}{2}\|x_k-x_{k+1}\|^2\\
            &\geq (1-\lambda)\phi^\alpha(y_k)-(1-\lambda)^{k+1}{\Delta} +\lambda\gamma_k(x_{k+1}) + \frac{\alpha(1-\lambda)}{2}\|x_k-x_{k+1}\|^2\\
            &\stackrel{\eqref{def:gamma_3avg}}\geq (1-\lambda)\gamma_k(y_k)-(1-\lambda)^{k+1}{\Delta} +\lambda\gamma_k(x_{k+1}) + \frac{\alpha(1-\lambda)}{2}\|x_k-x_{k+1}\|^2
            \end{align*}
            where second line follows by Lemma~\ref{lem:model_props}(b), the third by the inductive hypothesis and the fourth by the convexity of $f$. Then, further applying the convexity of $\gamma_k$ and~\eqref{def:y_3avg}, we obtain
            \begin{align*}
            \min_{x\in\R^n}\Gamma_{k+1}(x)
            &\geq \gamma_k(y_{k+1})-(1-\lambda)^{k+1}{\Delta}+ \frac{\alpha(1-\lambda)}{2}\|x_k-x_{k+1}\|^2\\
            &\stackrel{\eqref{def:tx_taa}\eqref{def:y_3avg}}= \gamma_k(y_{k+1})-(1-\lambda)^{k+1}{\Delta}+ \frac{\alpha(1-\lambda)}{2\lambda^2}\|y_{k+1}-\tx_{k+1}\|^2\\
            &\stackrel{\eqref{eq:lambda_alpha}}=\gamma_k(y_{k+1})-(1-\lambda)^{k+1}\Delta + \frac{L}{2}\|y_{k+1}-\tx_{k+1}\|^2\\
            &\stackrel{\eqref{def:gamma_3avg}}\geq \phi^\alpha(y_{k+1})-(1-\lambda)^{k+1}\Delta,
        \end{align*}
         where the first equality follows by the updates in~\eqref{def:tx_taa} and~\eqref{def:y_3avg}, the second by~\eqref{eq:lambda_alpha}, and the final line by the $L$-smoothness of $f$ and the definition of $\gamma_k$ in \eqref{def:gamma_3avg}. We therefore conclude the proof.
    \end{proof}
    \textit{Remark:} The key step in the previous proof was the identity $\lam(x_{k+1}-x_k)=y_{k+1}-\tx_{k+1}$ resulting from the geometric similarity illustrated in Figure~\ref{fig:AggGCG_taa_geometry}(b).

    Using the upper bound in Proposition~\ref{prop:single_step_3avg}, we have (1) a linear convergence rate for the primal-dual certificate $(y_k,\Gamma_k)$ and (2) a corresponding iteration-complexity for finding an $\varepsilon$-solution to~\eqref{def:probIntro}.
    
    \noindent\textbf{Proof of Theorem~\ref{thm:taa_cmplx}}: 
    Applying Proposition~\ref{prop:single_step_3avg} and rearranging terms, we have
        \begin{equation}
            \phi^\alpha(y_{k})-\min_{x\in\R^n}\Gamma_{k}(x)\leq (1-\lambda)^{k}\Delta.
        \end{equation}
        It follows from the definition of $\lam$ in \eqref{eq:x0_and_lambda} that $1-\lambda\leq \bigl(1+{\frac{\sqrt\alpha}{2\sqrt L}}\bigr)^{-2}$, which together with the above inequality implies that
        \begin{equation}
            \phi^\alpha(y_{k})-\min_{x\in\R^n}\Gamma_{k}(x)\leq \frac{\Delta}{\left(1+\frac{\sqrt{\alpha}}{2\sqrt{L}}\right)^{2k}}.
        \end{equation}
        We therefore obtain $\phi^\alpha(y_k)-\min_{x\in\R^n}\Gamma_k(x)\leq\varepsilon/2$ in 
        \begin{equation}
            k=\mathcal{O}\left(1+\sqrt{\frac{L}{\alpha}}\log\left(\frac{\phi^\alpha(y_0)-\Gamma_0(x_0)}{\varepsilon}\right)\right)
        \end{equation}
        iterations. Choosing $\alpha=\varepsilon/(2M)$ and applying Definition~\ref{def:pd_cert} gives the result.\QEDA
\subsection{Analysis of GEM}
    In this subsection, we analyze the non-asymptotic behavior of GEM for solving~\eqref{eq:fenchel_rockafellar_dual}. We begin with standard technical lemmas for accelerated methods.
    \begin{lemma}\label{lem:gem_technical}
        For all $k\geq 0$, the following hold
        \begin{itemize}
            \item[{\rm a)}] $ L^{-1}A_k= \alpha^{-1} \tau_k$;
            \item[{\rm b)}] $a_k^2=\tau_kA_{k+1}=\tau_{k+1}A_k$;
            \item[{\rm c)}] $A_{k}\geq\left(1+\frac{\sqrt\alpha}{2\sqrt{L}}\right)^{2k}$.
        \end{itemize}
    \end{lemma}
    \begin{proof}
        \noindent {a)} Note that expanding the recursion for $\tau_k$ in~\eqref{eq:gem_scalars} from $0$ to $k-1$, we have
        \begin{equation}\label{eq:tau_k_closed_form}
            \alpha^{-1}\tau_{k} = \alpha^{-1}\tau_0+L^{-1}\sum_{i=0}^{k-1}a_i= \alpha^{-1}\tau_0+L^{-1}(A_k-A_0).
        \end{equation}
        Using $\alpha^{-1}\tau_0=L^{-1}= L^{-1}A_0$ gives,
        \[
            \alpha^{-1}\tau_{k}=\alpha^{-1}\tau_0-L^{-1} A_0+L^{-1} A_k= L^{-1} A_k.
        \]
        \noindent b) The expression for $a_k$ in~\eqref{eq:gem_scalars} gives
        \[
        a_k^2=\tau_ka_k+\tau_kA_k\stackrel{\eqref{eq:gem_hatx}}=\tau_kA_{k+1}=(\alpha L^{-1}A_k)(\alpha^{-1}L\tau_{k+1})=A_{k}\tau_{k+1},
        \]
        where the third equality follows from part (a).

        \noindent c) From~\eqref{eq:gem_scalars}, we have
        \begin{equation}
            a_k\geq \frac{\tau_k}{2}+\sqrt{\tau_kA_k},
        \end{equation}
        which implies that
        \[
            A_{k+1}=a_k+A_{k}\geq \frac{\tau_k}{2}+\sqrt{\tau_kA_k}+A_k\geq \left(\sqrt{A_k}+{\frac{\sqrt{\tau_k}}{2}}\right)^2
            = A_k\left(1+\frac{\sqrt{\alpha}}{2\sqrt{L}}\right)^2,
        \]
        where the last equality follows by part (a). Expanding the inequality from $0$ to $k$ and using $A_0=1$ yields the claim.
    \end{proof}

    With the necessary technical lemmas established, we begin by examining the optimality conditions of the extrapolated update~\eqref{def:g_gem}.
    \begin{lemma}\label{lem:optim_gem}
        Define $\gamma_k(\cdot):=\inner{-\hat v_k}{\cdot}+f^*(\cdot)$. Then, for all $g\in\dom f^*$, we have
        \begin{equation}\label{eq:single_step_lin_gem}
            A_k\gamma_k(z_k)+a_k\gamma_k(g)+ \frac{\tau_k}{\alpha}\D_{\nu^*}(g\|g_k)\geq A_{k+1}\gamma_k(z_{k+1})+\frac{\tau_k}{\alpha}\D_{\nu^*}(g_{k+1}\|g_k)+\frac{\tau_{k+1}}{\alpha}\D_{\nu^*}(g\|g_{k+1})
        \end{equation}
    \end{lemma}
    \begin{proof}
        Noting that $f^*$ is $(1/L)$-strongly convex relative to $\nu^*$, and applying Lemma~\ref{lem:three_point} to the $g_{k+1}$ update in~\eqref{def:g_gem},  with $\Phi(\cdot):=a_k \gamma_k(\cdot)$, $\omega:=\nu^*$, $\beta=\tau_k/\alpha$ and $\mu=a_k/L$, we have for all $g\in\dom f^*$,
        \begin{equation*}
            a_k\gamma_k(g_{k+1})+\frac{\tau_k}{\alpha}\D_{\nu^*}(g_{k+1}\|g_k)+\alpha^{-1}\left(\frac{\alpha a_k}{L}+\tau_k\right)\D_{\nu^*}(g\|g_{k+1})\leq a_k\gamma_k(g)+\frac{\tau_k}{\alpha}\D_{\nu^*}(g\|g_k).
        \end{equation*}
        Adding $A_{k}\gamma_k(z_k)$ to both sides gives
        \begin{align*}
            &A_{k}\gamma_k(z_k)+a_k\gamma_k(g)+\frac{\tau_k}{\alpha}\D_{\nu^*}(g\|g_k)  \\
            \geq&a_k\gamma_k(g_{k+1})+A_{k}\gamma_k(z_k)+\frac{\tau_k}{\alpha}\D_{\nu^*}(g_{k+1}\|g_k)+\alpha^{-1}\left(\frac{\alpha a_k}{L}+\tau_k\right)\D_{\nu^*}(g\|g_{k+1})\\
            \stackrel{\eqref{def:xz_gem}}\geq&A_{k+1}\gamma_k(z_{k+1})+\frac{\tau_k}{\alpha}\D_{\nu^*}(g_{k+1}\|g_k)+\frac{\tau_{k+1}}{\alpha}\D_{\nu^*}(g\|g_{k+1}),
        \end{align*}
        where the second inequality follows by the convexity of $\gamma_k(\cdot)$ and the definition of $\tau_{k+1}$ in~\eqref{eq:gem_scalars}.
    \end{proof}
    
    The following lemma gives an equivalent but more convenient form of \eqref{eq:single_step_lin_gem} in Lemma \ref{lem:optim_gem}.
    \begin{lemma}\label{lem:single_step_ub_psi}
        For all iterations $k\geq 0$ of Algorithm~\ref{alg:GEM-dual}, the following inequality holds
        \begin{align}
            \nonumber A_k\left[\ell_{(h^\alpha)^*}(z_k;z_{k+1})+f^*(z_k)\right]+\frac{\tau_k}{\alpha}\D_{\nu^*}(g\|g_k)-\frac{\tau_k}{\alpha}\D_{\nu^*}(g_{k+1}\|g_k) - \frac{\tau_{k+1}}{\alpha}\D_{\nu^*}(g\|g_{k+1})
        \\\geq A_{k+1}\psi^\alpha(z_{k+1})+a_k\left[h^\alpha(v_{k+1})+\inner{v_{k+1}}{g}-f^*(g)\right] + a_k\inner{v_{k+1}-\hat v_k}{g_{k+1}-g}\label{ineq:single_step_ub_psi}.
        \end{align}
    \end{lemma}
    \begin{proof}
        Using the definition of $\gamma_k$, we observe that for all $g\in\dom f^*$, the following identity holds
        \begin{align*}
            \gamma_k(g)&=\inner{-\hat v_k}{g}+f^*(g)=\inner{v_{k+1}-\hat v_k}{g}+\inner{-v_{k+1}}{g-z_{k+1}}-\inner{v_{k+1}}{z_{k+1}}+f^*(g)\\
            &\stackrel{\eqref{def:xz_gem}}=\inner{v_{k+1}-\hat v_k}{g}+\ell_{(h^\alpha)^*}(g;z_{k+1})-(h^\alpha)^*(-z_{k+1})-\inner{v_{k+1}}{z_{k+1}}+f^*(g),
        \end{align*}
        where the final equality holds by $v_{k+1}=\nabla (h^\alpha)^*(-z_{k+1})$ in \eqref{def:xz_gem} and the definition of $\ell_{(h^\alpha)^*}(\cdot;z)$ in \eqref{def:ell}. Expanding each instance of $\gamma_k(\cdot)$ in~\eqref{eq:single_step_lin_gem}, we obtain
        \begin{align*}
            &A_k\left[\inner{v_{k+1}-\hat v_k}{z_k}+\ell_{(h^\alpha)^*}(z_k;z_{k+1})-(h^\alpha)^*(-z_{k+1})-\inner{v_{k+1}}{z_{k+1}}+f^*(z_k)\right]\\
            &+a_k\left[\inner{v_{k+1}-\hat v_k}{g}+\ell_{(h^\alpha)^*}(g;z_{k+1})-(h^\alpha)^*(-z_{k+1})-\inner{v_{k+1}}{z_{k+1}}+f^*(g)\right]+\frac{\tau_k}{\alpha} \D_{\nu^*}(g\|g_k)\\
            \geq& A_{k+1}\left[-\inner{v_{k+1}}{z_{k+1}}+\inner{v_{k+1}-\hat v_k}{z_{k+1}}+f^*(z_{k+1})\right]+\frac{\tau_{k}}{\alpha}\D_{\nu^*}(g_{k+1}\|g_k)+\frac{\tau_{k+1}}{\alpha}\D_{\nu^*}(g\|g_{k+1}).
        \end{align*}
        Rearranging yields
        \begin{align*}
            &A_k\left[\ell_{(h^\alpha)^*}(z_k;z_{k+1})+f^*(z_k)\right]+\frac{\tau_k}{\alpha} \D_{\nu^*}(g\|g_k)-\frac{\tau_{k+1}}{\alpha}\D_{\nu^*}(g\|g_{k+1})-\frac{\tau_k}{\alpha}\D_{\nu^*}(g_{k+1}\|g_k)\\
            \geq& A_{k+1}\psi^\alpha(z_{k+1})+a_k\left[- \ell_{(h^\alpha)^{*}}(g;z_{k+1})-f^*(g)\right]+\inner{v_{k+1}-\hat v_k}{A_{k+1}z_{k+1}-A_kz_k-a_kg}\\
            = & A_{k+1}\psi^\alpha(z_{k+1})+a_k\left[\inner{v_{k+1}}{g} + h^\alpha(v_{k+1})-f^*(g)\right]+\inner{v_{k+1}-\hat v_k}{A_{k+1}z_{k+1}-A_kz_k-a_kg},
        \end{align*}
        where the equality follows from the fact that
        \begin{align*}
           \ell_{(h^\alpha)^*}(g;z_{k+1})\stackrel{\eqref{def:ell}}=&(h^\alpha)^*(-z_{k+1}) + \inner{-\nabla (h^\alpha)^*(-z_{k+1})}{g-z_{k+1}}\\
           \stackrel{\eqref{def:xz_gem}}=&(h^\alpha)^*(-z_{k+1}) - \inner{v_{k+1}}{-z_{k+1}}-\inner{v_{k+1}}{g}\\
           \stackrel{Lemma~\ref{lem:beck_conj_pair}}=&  -h^\alpha(v_{k+1})-\inner{v_{k+1}}{g}.
        \end{align*}
        Finally, we note that
        \begin{equation}
            A_{k+1}z_{k+1}-A_kz_k\stackrel{\eqref{def:xz_gem}}=a_kg_{k+1},
        \end{equation}
        which completes the proof.
    \end{proof}
    Next, we state a simple consequence of our choice of $a_k$, $A_k$ and the smoothness of $(h^\alpha)^*$.
    \begin{lemma}\label{lem:hlam_smooth_ub}
        For all iterations $k\geq 0$ of Algorithm~\ref{alg:GEM-dual}, the following inequality holds
        \begin{equation}\label{ineq:hlam_smooth_ub}
            \begin{split}
                A_k\psi^\alpha(z_k)-\frac{\alpha a_k^2}{2\tau_{k+1}}\|v_{k+1}-v_k\|^2\geq A_k\left[\ell_{(h^\alpha)^*}(z_k;z_{k+1})+f^*(z_k)\right].
            \end{split}
        \end{equation}
    \end{lemma}
    \begin{proof}
        By the $(1/\alpha)$-smoothness of $(h^\alpha)^*$ (see~\cite[Theorem 5.8(iii)]{Beck2017}), we have
        \begin{align}
            \nonumber\ell_{(h^\alpha)^*}(z_k;z_{k+1})&\leq (h^\alpha)^*(-z_k)-\frac{\alpha}{2}\|\nabla (h^\alpha)^*(-z_k)-\nabla (h^\alpha)^*(-z_{k+1})\|^2\\
            &\stackrel{\eqref{def:xz_gem}}=(h^\alpha)^*(-z_k)-\frac{\alpha}{2}\|v_k-v_{k+1}\|^2.\label{ineq:hlam_smooth}
        \end{align}
        It thus follows from Lemma~\ref{lem:gem_technical}(b) that
        \begin{equation*}
            A_k\psi^\alpha(z_k)-\frac{\alpha a_k^2}{2 \tau_{k+1}}\|v_k-v_{k+1}\|^2= A_k\left[\psi^\alpha(z_k)-\frac{\alpha}{2}\|v_k-v_{k+1}\|^2\right]  \stackrel{\eqref{ineq:hlam_smooth}}\geq A_k\left[f^*(z_k)+\ell_{(h^\alpha)^*}(z_k;z_{k+1})\right].
        \end{equation*}
    \end{proof}
    
    The following lemma provides a convenient algebraic identity resulting from the definition of the extrapolation point $\hat v_{k}$.
    \begin{lemma}\label{lem:sum_grad_extrap}
        For every $k\geq 0$, define
        \begin{equation}
            s_k:=a_k(v_{k+1}-v_k)\label{def:sk}
        \end{equation}
        with $s_{-1}:=0$. Then, for all $g\in\R^n$ we have
        \begin{equation}\label{eq:sum_grad_extrap}
            \sum_{i=0}^{k-1}a_i\inner{v_{i+1}-\hat v_i}{g_{i+1}-g}=\inner{s_{k-1}}{g_k-g}-\sum_{i=1}^{k-1}\inner{s_{i-1}}{g_{i+1}-g_i}.
        \end{equation}
    \end{lemma}
    \begin{proof}
        By the definitions of $s_k$ and $\hat v_k$ in \eqref{def:sk} and \eqref{eq:gem_hatx}, respectively, we have
        \[
            s_k-s_{k-1}=a_k(v_{k+1}-v_k)-a_{k-1}(v_{k}-v_{k-1})
            \stackrel{\eqref{eq:gem_hatx}}=a_k(v_{k+1}-\hat v_{k}).
        \]
        Then, summing from $0$ to $k-1$ and using $s_{-1}=0$, we have
        \begin{align*}
            \sum_{i=0}^{k-1}a_i & \inner{v_{i+1}-\hat v_i}{g_{i+1}-g}=\sum_{i=0}^{k-1}\left(\inner{s_i}{g_{i+1}-g}-\inner{s_{i-1}}{g_{i+1}-g}\right)\\
            &=-\inner{s_{k-1}}{g}+\sum_{i=0}^{k-1}\inner{s_i-s_{i-1}}{g_{i+1}}
            =\inner{s_{k-1}}{g_k-g}-\sum_{i=1}^{k-1}\inner{s_{i-1}}{g_{i+1}-g_i}.
        \end{align*}
        We therefore prove the claim.
    \end{proof}

    Combining the previous lemmas, we are now ready to prove Theorem~\ref{thm:gem_cmplx}.

    \noindent\textbf{Proof of Theorem~\ref{thm:gem_cmplx}}: First, applying Lemmas~\ref{lem:single_step_ub_psi} and~\ref{lem:hlam_smooth_ub}, we obtain for all iterations $i\geq 0$ and all $g\in\dom f^*$
        \begin{align*}
            A_i\psi^\alpha(z_i)-\frac{\alpha a_i^2}{2\tau_{i+1}}\|v_{i+1}-v_i\|^2+\frac{\tau_i}{\alpha}\D_{\nu^*}(g\|g_i)-\frac{\tau_i}{\alpha}\D_{\nu^*}(g_{i+1}\|g_i) - \frac{\tau_{i+1}}{\alpha}\D_{\nu^*}(g\|g_{i+1})
        \\\stackrel{\eqref{ineq:single_step_ub_psi}\eqref{ineq:hlam_smooth_ub}}\geq A_{i+1}\psi^\alpha(z_{i+1})+a_i\left[h^\alpha(v_{i+1})+\inner{v_{i+1}}{g}-f^*(g)\right] + a_i\inner{v_{i+1}-\hat v_i}{g_{i+1}-g}.
        \end{align*}
        Summing from $i=0$ to $k-1$, rearranging, dropping the non-negative term $\D_{\nu^*}(g_1\|g_{0})$, and using $\tau_0=\alpha/L$ gives
        \begin{align}
            \nonumber A_0\psi^\alpha(z_0)&+L^{-1}\D_{\nu^*}(g\|g_0)
            \nonumber\geq A_{k}\psi^\alpha(z_k) +\frac{\tau_k}{\alpha}\D_{\nu^*}(g\|g_k)+ \sum_{i=0}^{k-1}a_i[h^\alpha(v_{i+1})+\inner{v_{i+1}}{g}-f^*(g)]\\
            \nonumber+&\frac{\alpha a_{k-1}^2}{2\tau_k}\|v_{k}-v_{k-1}\|^2+\sum_{i=0}^{k-1}a_i\inner{v_{i+1}-\hat v_i}{g_{i+1}-g}+\sum_{i=1}^{k-1}\left(\frac{\tau_i}{\alpha}\D_{\nu^*}(g_{i+1}\|g_{i}) + \frac{\alpha a_i^2}{2\tau_{i}}\|v_{i}-v_{i-1}\|^2\right)\\
            \nonumber\stackrel{\eqref{def:sk}\eqref{eq:sum_grad_extrap}}=&A_{k}\psi^\alpha(z_k) + \frac{\alpha}{2\tau_k}\|s_{k-1}\|^2 +\frac{\tau_k}{\alpha}\D_{\nu^*}(g\|g_k)+ \sum_{i=0}^{k-1}a_i[h^\alpha(v_{i+1})+\inner{v_{i+1}}{g}-f^*(g)]\\
            \nonumber&+\inner{s_{k-1}}{g_k-g}+\sum_{i=1}^{k-1}\left(\frac{\tau_i}{\alpha}\D_{\nu^*}(g_{i+1}\|g_{i})- \inner{s_{i-1}}{g_{i+1}-g_i}+ \frac{\alpha}{2\tau_{i}}\|s_{i-1}\|^2\right),
            \end{align}
            where the equality follows by Lemma~\ref{lem:sum_grad_extrap} and the definition of $s_k$. Then, using the 1-strong convexity of $\nu^*$, we obtain
            \begin{align}
           \nonumber A_0\psi^\alpha(z_0)+L^{-1}\D_{\nu^*}(g\|g_0)\stackrel{\eqref{def:sk}}\geq& A_{k}\psi^\alpha(z_k) + \sum_{i=0}^{k-1}a_i[h^\alpha(v_{i+1})+\inner{v_{i+1}}{g}-f^*(g)]\\
            \nonumber&+\frac{\tau_k}{2\alpha}\|g-g_k\|^2+\inner{s_{k-1}}{g_k-g} + \frac{\alpha }{2\tau_k}\|s_{k-1}\|^2\\
            \nonumber &+\sum_{i=1}^{k-1}\left(\frac{\tau_i}{2\alpha}\|g_{i+1}-g_{i}\|^2- \inner{s_{i-1}}{g_{i+1}-g_i}+ \frac{\alpha}{2\tau_{i}}\|s_{i-1}\|^2\right)\\
            \geq &A_{k}\psi^\alpha(z_k) +  \sum_{i=0}^{k-1}a_i[h^\alpha(v_{i+1})+\inner{v_{i+1}}{g}-f^*(g)],\label{ineq:gem_sum_last}
        \end{align}
        where the final line follows by the Cauchy-Schwarz inequality.
        It follows from Lemma~\ref{lem:beck_conj_pair}, \eqref{def:xz_gem}, and the definition of $\ell_{(h^\alpha)^*}(\cdot;z)$ in \eqref{def:ell} that for any $i\geq 0$,
        \begin{align}
            \nonumber h^\alpha(v_{i})+\inner{v_{i}}{g} &= h^\alpha(v_{i})-\inner{v_i}{-z_i}+\inner{v_{i}}{g-z_{i}}\\
            \nonumber&\stackrel{\eqref{def:xz_gem}}=-(h^\alpha)^*(-z_{i})-\inner{-\nabla (h^\alpha)^*(-z_i)}{g-z_{i}}\\
        &\stackrel{\eqref{def:ell}}=-\ell_{(h^\alpha)^*}(g;z_i).\label{eq:conj_lin_algebra}
        \end{align}
        Applying this identity to the summation on the right-hand side of~\eqref{ineq:gem_sum_last}, we have
        \begin{align}
            \sum_{i=0}^{k-1}a_i[h^\alpha(v_{i+1})+\inner{v_{i+1}}{g}-f^*(g)]
            \stackrel{\eqref{eq:conj_lin_algebra}}=& -A_k\Biggl(\sum_{i=0}^{k-1}\frac{a_i}{A_{k}}[\ell_{(h^\alpha)^*}(g;z_{i+1})+f^*(g)] \\
            &+\frac{A_0}{A_k}\left(\ell_{(h^\alpha)^*}(g;z_{0})+f^*(g)\right)\Biggr) 
            + A_0(\ell_{(h^\alpha)^*}(g;z_{0})+f^*(g))\\
            \stackrel{\eqref{eq:conj_lin_algebra}}= &-A_{k}\Gamma_{k}^*(g) - A_0(h^\alpha(v_{0})+\inner{v_0}{g}-f^*(g)).\label{eq:duality_sum_gamma_gem}
        \end{align}
        where the first equality follows from~\eqref{eq:conj_lin_algebra} and the second by the definition of the model $\Gamma_k^*(\cdot)$ induced by $(z_0, \{z_{i+1}\}_{i=0}^{k-1}, \{a_i/A_{i+1}\}_{i=0}^{k-1})$ with $A_0=1$ and~\eqref{eq:conj_lin_algebra} applied in reverse.
        
        Applying~\eqref{eq:duality_sum_gamma_gem} to~\eqref{ineq:gem_sum_last} and maximizing both sides over $g\in\dom f^*$ gives
        \begin{align*}
            A_0(\psi^\alpha(z_0)+\phi^\alpha(v_0))+L^{-1}\max_{g\in\dom f^*}\D_{\nu^*}(g\|g_0)&\geq A_{k}\left(\psi^\alpha(z_k) -\min_{g\in\R^n}\Gamma_{k}^*(g)\right),
        \end{align*}
        where we use
        \begin{equation*}
            h^\alpha(v_{0})+\max_{g\in\dom f^*}\{\inner{v_{0}}{g}-f^*(g)\}= \phi^\alpha(v_0).
        \end{equation*}
        Rearranging, using Lemma~\ref{lem:gem_technical}(c), and noting that $A_0=1$ gives the certificate convergence
        \begin{equation*}
            \psi^\alpha(z_k)-\min_{g\in\R^n}\Gamma_k^*(g)\leq \frac{\phi^\alpha(v_0)+\psi^\alpha(z_0)+ DL^{-1}}{\left(1+\frac{\sqrt\alpha}{2\sqrt{ L}}\right)^{2k}}.
        \end{equation*}
        The number of iterations $k$ to obtain $\psi^\alpha(z_k)-\min_g\Gamma_k^*(g)\leq \varepsilon/2$ is therefore 
        
        \begin{equation}
            k=\mathcal{O}\left(1+\sqrt{\frac{L}{\alpha}}\log\left(\frac{\phi^\alpha(v_0)+\psi^\alpha(z_0)+ DL^{-1}}{\varepsilon}\right)\right)
        \end{equation}
        by standard analysis. The result follows by defining the dual certificate (analogous to Definition~\ref{def:pd_cert}) as the pair $(z_k,\Gamma_k^*)$
        and the choice $\alpha=\varepsilon/(2M)$.\QEDA


\bibliographystyle{alpha}
\bibliography{ref}
\end{document}